\documentclass[a4paper,10pt]{siamltex} 
\usepackage{amsmath,amssymb,amsfonts,graphicx,epsfig,color,url}
\usepackage{latexsym,mathrsfs,cite} 
\newtheorem{Remark}[theorem]{Remark}

\definecolor{colorJRblue}{rgb}{0.,0.,1.}
\definecolor{colorJRred}{rgb}{1.,0.,0.}
\definecolor{colorESpurple}{rgb}{1.,0.,1.}

\def\C{\mathbb{C}}
\def\R{\mathbb{R}}
\def\E{\mathbb{E}}
\def\N{\mathbb{N}}
\def\L{\mathcal{L}}
\def\X{\mathcal{X}}
\def\V{\mathcal{V}}
\def\calT{\mathcal{T}}

\def\Lspace{L}
\def\Hspace{H}
\def\BC{\mathrm{BC}}
\def\BUC{\mathrm{BUC}}

\def\B{\hat B}
\def\spec{\mathrm{spec}}
\def\per{\mathrm{per}}
\def\rme{\mathrm{e}}
\def\rmi{\mathrm{i}}

\def\rmd{\mathrm{d}}

\def\bu{\overline{u}}
\def\bw{\overline{w}}
\def\bv{\overline{v}}
\def\v{\mathbf{u}}

\def\typeout#1{\message{^^J}\message{#1}\message{^^J}}
\newif\ifSRCOK \SRCOKtrue
\newcount\PAGETOP
\newcount\LASTLINE
\global\PAGETOP=1
\global\LASTLINE=-1
\def\EJECT{\SRC\eject}
\def\WinEdt#1{\typeout{:#1}}
\gdef\MainFile{\jobname.tex}
\gdef\CurrentInput{\MainFile}
\newcount\INPSP
\global\INPSP=0
\def\SRC{\ifSRCOK%
  \ifnum\inputlineno>\LASTLINE%
    \ifnum\LASTLINE<0%
      \global\PAGETOP=\inputlineno%
    \fi%
    \global\LASTLINE=\inputlineno%
    \ifnum\INPSP=0%
      \ifnum\inputlineno>\PAGETOP%
        
      \fi%
    \else%
      
    \fi%
  \fi%
\fi}
\def\PUSH#1{%
\SRC%
\ifnum\INPSP=0 \global\let\INPSTACKA=\CurrentInput \else%
\ifnum\INPSP=1 \global\let\INPSTACKB=\CurrentInput \else%
\ifnum\INPSP=2 \global\let\INPSTACKC=\CurrentInput \else%
\ifnum\INPSP=3 \global\let\INPSTACKD=\CurrentInput \else%
\ifnum\INPSP=4 \global\let\INPSTACKE=\CurrentInput \else%
\ifnum\INPSP=5 \global\let\INPSTACKF=\CurrentInput \else%
               \global\let\INPSTACKX=\CurrentInput \fi\fi\fi\fi\fi\fi%
\gdef\CurrentInput{#1}%
\WinEdt{<+ \CurrentInput}%
\global\LASTLINE=0%
\ifSRCOK\fi%
\global\advance\INPSP by 1}
\def\POP{%
\ifnum\INPSP>0 \global\advance\INPSP by -1  \fi%
\ifnum\INPSP=0 \global\let\CurrentInput=\INPSTACKA \else%
\ifnum\INPSP=1 \global\let\CurrentInput=\INPSTACKB \else%
\ifnum\INPSP=2 \global\let\CurrentInput=\INPSTACKC \else%
\ifnum\INPSP=3 \global\let\CurrentInput=\INPSTACKD \else%
\ifnum\INPSP=4 \global\let\CurrentInput=\INPSTACKE \else%
\ifnum\INPSP=5 \global\let\CurrentInput=\INPSTACKF \else%
               \global\let\CurrentInput=\INPSTACKX \fi\fi\fi\fi\fi\fi%
\WinEdt{<-}%
\global\LASTLINE=\inputlineno%
\global\advance\LASTLINE by -1%
\SRC}
\def\INPUT#1{\relax}

\let\originalxxxeverypar\everypar
\newtoks\everypar
\originalxxxeverypar{\the\everypar\expandafter\SRC}
\everymath\expandafter{\the\everymath\expandafter\SRC}
\output\expandafter{\expandafter\SRCOKfalse\the\output}

\newif\ifSRCOK \SRCOKtrue
\newcount\PAGETOP
\newcount\LASTLINE
\global\PAGETOP=1
\global\LASTLINE=-1
\gdef\MainFile{\jobname.tex}
\gdef\CurrentInput{\MainFile}
\newcount\INPSP
\global\INPSP=0
\def\EJECT{\SRC\eject}
\def\WinEdt#1{\typeout{:#1}}
\def\SRC{\ifSRCOK%
  \ifnum\inputlineno>\LASTLINE%
    \ifnum\LASTLINE<0%
      \global\PAGETOP=\inputlineno%
    \fi%
    \global\LASTLINE=\inputlineno%
    \ifnum\INPSP=0%
      \ifnum\inputlineno>\PAGETOP%
      \fi%
    \else%
    \fi%
  \fi%
\fi}
\def\PUSH#1{%
\SRC%
\ifnum\INPSP=0 \global\let\INPSTACKA=\CurrentInput \else%
\ifnum\INPSP=1 \global\let\INPSTACKB=\CurrentInput \else%
\ifnum\INPSP=2 \global\let\INPSTACKC=\CurrentInput \else%
\ifnum\INPSP=3 \global\let\INPSTACKD=\CurrentInput \else%
\ifnum\INPSP=4 \global\let\INPSTACKE=\CurrentInput \else%
\ifnum\INPSP=5 \global\let\INPSTACKF=\CurrentInput \else%
               \global\let\INPSTACKX=\CurrentInput \fi\fi\fi\fi\fi\fi%
\gdef\CurrentInput{#1}%
\WinEdt{<+ \CurrentInput}%
\global\LASTLINE=0%
\ifSRCOK\fi%
\global\advance\INPSP by 1}
\def\POP{%
\ifnum\INPSP>0 \global\advance\INPSP by -1  \fi%
\ifnum\INPSP=0 \global\let\CurrentInput=\INPSTACKA \else%
\ifnum\INPSP=1 \global\let\CurrentInput=\INPSTACKB \else%
\ifnum\INPSP=2 \global\let\CurrentInput=\INPSTACKC \else%
\ifnum\INPSP=3 \global\let\CurrentInput=\INPSTACKD \else%
\ifnum\INPSP=4 \global\let\CurrentInput=\INPSTACKE \else%
\ifnum\INPSP=5 \global\let\CurrentInput=\INPSTACKF \else%
               \global\let\CurrentInput=\INPSTACKX \fi\fi\fi\fi\fi\fi%
\WinEdt{<-}%
\global\LASTLINE=\inputlineno%
\global\advance\LASTLINE by -1%
\SRC}
\def\INPUT#1{\relax}
\let\OldINCLUDE=\include
\def\include#1{
\EJECT%
\PUSH{#1.tex}%
\OldINCLUDE{#1}%
\POP}

\let\originalxxxeverypar\everypar
\newtoks\everypar
\originalxxxeverypar{\the\everypar\expandafter\SRC}
\everymath\expandafter{\the\everymath\expandafter\SRC}
\let\zzzxxxbibliography=\bibliography
\def\bibliography#1{\PUSH{\jobname.bbl}\zzzxxxbibliography{#1}\POP}
\output\expandafter{\expandafter\SRCOKfalse\the\output}

\begin{document}

\title{Quasilinear parabolic reaction-diffusion systems: user's guide to well-posedness, spectra and stability of  travelling waves}

\author{ M.\ Meyries\thanks{Martin-Luther-Universit\"at Halle-Wittenberg, Institut f\"ur Mathematik, 06099 Halle, Germany; martin.meyries@mathematik.uni-halle.de} \and J.D.M. Rademacher\thanks{Universit\"at Bremen, Fachbereich Mathematik, Postfach 33 04 40, 28359 Bremen, Germany; rademach@math.uni-bremen.de} \and E.\ Siero\thanks{Mathematisch Instituut, Universiteit Leiden, P.O. Box 9512, 2300 RA Leiden, the Netherlands; esiero@math.leidenuniv.nl} }

\maketitle

\begin{abstract}
This paper is concerned with quasilinear parabolic reaction-diffusion-advection systems on extended domains.  Frameworks for well-posedness  in Hilbert spaces and spaces of continuous functions are presented, based on known results using maximal regularity. It is shown that spectra of travelling waves on the line are meaningfully given by the familiar tools for semilinear equations, such as dispersion relations, and basic connections of spectra to stability and instability are considered. In particular, a principle of linearized orbital instability for manifolds of equilibria is proven. Our goal is to provide easy access for applicants to these rigorous aspects.  As a guiding example the Gray-Scott-Klausmeier model for vegetation-water interaction is considered in detail. 
\end{abstract}



\section{Introduction}\label{s:intro}~
In this paper we present rigorous frameworks for well-posedness, spectra and nonlinear stability of travelling wave solutions (pulses, fronts and wavetrains) of quasilinear parabolic reaction-diffusion systems of the form
\begin{equation} \label{e:quasi}
u_t = (a(u)u_x)_x  + f(u,u_x), \qquad t>0, \qquad x\in \R,
\end{equation}
with unknown $u(t,x)\in \R^N$. The nonlinearities $a,f$ are smooth and $a(u)\in \R^{N\times N}$ is strongly elliptic in the domain of interest, but does not have to be symmetric. We further consider a variant of \eqref{e:quasi} in higher space dimensions $x\in \R^n$ up to $n =3$.  The nonlinearities may also depend explicitly on $x$ in an appropriate way. 

Quasilinear reaction-diffusion systems arise as models in various contexts due to nonlinear fluxes, density dependent diffusion, self or cross diffusion, see e.g. \cite{Amann}. For pattern formation problems it is natural to consider an extended domain and to neglect the influence of boundary conditions. Travelling waves, i.e., solutions of \eqref{e:quasi} constant in a co-moving frame $\xi = x-ct$ with speed $c\in \R$ having constant or periodic asymptotic states, are among the simplest interesting reaction-diffusion patterns and are observed for different types of quasilinear systems, see, e.g., \cite{Maini, Mey, NagaiIkeda, WuZhao, Holzer, KumarHorsthemke, Meron}. 

For semilinear parabolic problems on the line it is well-known that e.g. $H^1$ or $\text{BUC}^1$ are suitable phase spaces for well-posedness in a perturbative setting \cite{CD, Henry}. The corresponding spectrum of the linearization is characterized in terms of the dispersion relation and the Evans function \cite{Evans,BjornSurvey}. In some situations, in particular when the essential spectrum does not touch the imaginary axis, nonlinear (orbital) stability of a wave can directly be deduced by a principle of linearized stability \cite{Henry,Schneider}. 

For quasilinear models an analogous unified framework for well-\-posed\-ness, spectra and stability of waves seems less known. 
It seems that  the majority of concrete well-posedness results in the literature concerns bounded domains. Moreover, when the general results are formulated abstractly or under abstract conditions, an applicant needs to search for suitable function spaces and verify hypotheses that lead astray (even though some examples provide guidelines).

However, the spectrum of the linearization in a travelling wave can only be meaningfully determined based on a well-posedness setting. For instance, a Turing-instability determined via the usual dispersion relation lacks a basis without a consistent phase space. Conveniently, the pattern forming nature of a Turing-instability can be identified ad hoc since the existence of travelling wave patterns is an ODE problem. Well-posedness is, however, required to prove that a spectrally unstable solution indeed is unstable under the nonlinear evolution. Such a result then justifies the computation of stability boundaries by the spectrum as in \cite{RSS,Sjors} (see also \S\ref{s:gkgs}). 

The purpose of this paper is to present rigorous settings for quasilinear parabolic problems in the travelling wave context as described above.  We aim for a presentation accessible  to applicants, in the spirit of \cite{BjornSurvey, CD, Henry} for semilinear problems. To this end we bring together and apply to \eqref{e:quasi} mostly abstract results from the different fields involved in well-posedness, spectra and stability.

There are several abstract settings for well-posedness of general quasilinear para\-bolic problems avaliable in the literature (see \cite{Amann, Amann3, CS, GR, Kato,  KPW10, Pruess, Lunardi1, Yagi}, and \cite{Amann2} as well as \S\ref{s:outlook} for a selective overview). These have advantages and disadvantages depending on the present context, and the geometric (qualitative) theory is more or less developed in each case.  On the other hand,  solutions may be constructed by fixed point arguments taylor-made for the issues under investigation (e.g. \cite{Zumbrun}). The (real) viscous conservation laws are an important and well studied class of quasilinear problems, where well-posedness results exploit the additional structure \cite{Kaw}. We refer to the survey \cite{ZumSurv} and the references therein. 

Our focus lies on the approach of \cite{CL, KPW10, Pruess} based on maximal $L^p$-regularity, but we also highlight the approach of \cite{Lunardi1} based on maximal H\"older regularity. Besides reaction-diffusion problems, the approach of \cite{CL, KPW10, Pruess} and its extensions apply successfully to the local theory of free boundary problems and to general parabolic problems with nonlinear boundary conditions. Here the geometric theory is well-developed and still advances, especially for the needs in the context of free boundary problems. The approach of \cite{Lunardi1} also applies to fully nonlinear problems.  

Recently, in \cite{PSZ09, PSZ09-2} the principle of linearized orbital stability with asymptotic phase for manifolds of equilibria has been established in the quasilinear case, for any sufficiently strong well-posedness setting (see e.g. \cite[Section 5.1]{Henry} for the semilinear case). It in particular applies to the orbital stability of pulses and fronts for \eqref{e:quasi} in both approaches mentioned before.  The conclusion from arbitrary unstable spectrum to nonlinear orbital instability of a manifold of equilibria does not seem to exist in the literature. Refining arguments from \cite[Theorem 5.1.5]{Henry} for single equilibria, we close this gap in the present paper. This might be of interest also in other contexts, where families of equilibria occur.

In more detail, our considerations may be summarized as follows. 
\begin{itemize}
\item In one space dimension, $x\in \R$, a possible phase space for the evolution under \eqref{e:quasi} of localized perturbations from travelling wave and other pattern type solutions  is the Sobolev space $H^2$ (Theorem \ref{t:quasiRDSpert}).  For non-localized perturbations $\BUC^2$ is a possible phase space (Theorem \ref{t:lun-1D}).
\item For space dimensions $x\in \R^n$ with $n\leq 3$ other possible phase spaces are certain Besov spaces, (real) interpolating between $L^2$ and the Sobolev space $H^2$ (Theorem \ref{t:quasiRDS-Rn}). Here the linearization can directly be considered on $L^2$.
\item The `spatial dynamics' spectral theory developed for semilinear parabolic systems on the line applies also in the quasilinear case, which allows to compute the spectrum of travelling waves in a familiar way (see \S\ref{s:specwtSob}). In particular, the spectrum is independent of the chosen setting (Proposition \ref{prop:spex-indie}).
\item The well-known nonlinear stability result with asymptotic phase for travelling waves with simple zero eigenvalue applies in  these settings (Proposition \ref{prop:stability}, as a direct consequence of \cite{PSZ09, PSZ09-2}). 
\item Without assuming a spectral gap or an unstable eigenvalue, it is shown that an unstable spectrum implies orbital instability of pulses and fronts (Theorem \ref{t:orb-instab}) and instability of wavetrains (Proposition \ref{prop:instab-wavetrain}). Here we rely on a general result on orbital instability of manifolds of equilibria (Lemma \ref{l:Henry-lem}).
\end{itemize}

We emphasize that  the divergence form \eqref{e:quasi} is only assumed in view of applications. In a smooth setting, the equation $u_t = a(u)u_{xx} + f(u,u_x)$ can be cast into divergence form by a suitable redefinition of $a$ and $f$.

We believe that also the more general results in \cite{SanSchModSpec} on spectra of modulated travelling waves carry over to the quasilinear case, but  we do not enter into details here. Also the nonlinear stability of wavetrains is not considered.  This is a delicate issue since zero always lies in the essential spectrum. Hence, the best one can hope for is heat-equation-like decay. Under certain assumptions this has been established for the semilinear reaction-diffusion case in \cite{Schneider,DSSS}. A special quasilinear case, more precisely the quasilinear IBL model, is considered in \cite{HUS12}. Also for viscous shocks the spectrum touches the origin and stability in weighted spaces can be established. We refer to \cite{ZuHo98}, the survey \cite{ZumSurv} and the references therein, as well as to \cite{BSZ10} for more recent results.

In \S\ref{s:gkgs} we illustrate our general considerations by means of the Gray-Scott-\-Klaus\-meier ve\-ge\-ta\-tion-water interaction model \cite{Klausmeier}, for $x\in \R$ given  by
\begin{equation} \label{e:gkgs}
\begin{aligned}
 w_t=&(w^2)_{xx}+Cw_x+A(1-w)-wv^2,\\
 v_t=&Dv_{xx}-Bv+wv^2,
\end{aligned}
\end{equation}  
with constants $A,B\geq 0$, $C\in \R$ and $D >0$. This system is the original motivation for the present study. It is quasilinear due to the porous medium term $(w^2)_{xx} = 2(ww_{xx} + (w_x)^2)$ and is therefore parabolic only in the regime $w> 0$, in which \eqref{e:gkgs}  supports a large family of travelling waves (see \cite{Sjors} and \S\ref{s:gkgs}). 

This paper is organized as follows. In \S\ref{s:gkgs} we expand the discussion of \eqref{e:gkgs} and illustrate the applicability of the subsequent general results. \S\ref{s:wellHil} is devoted to different well-posedness setting results for \eqref{e:quasi}, and in \S\ref{s:spec} the spectrum of the linearization in travelling waves is treated. The connection to nonlinear stability and instability is considered in \S\ref{s:nl-stab}.  For the sake of self-containedness we prove some technical results in the appendix.

\textbf{Notation.} All Banach spaces are real, and we consider complexifications if necessary. We write $\mathscr L(X_1,X_0)$ for the bounded linear operators between Banach spaces $X_0, X_1$, and $\mathscr L(X_0) = \mathscr L(X_0,X_0)$.

\textbf{Acknowledgement.} J.R. and E.S. acknowledge support by the Complexity program of the Dutch research fund (NWO). J.R. is grateful for the support of the NWO cluster NDNS+ and his previous employer, Centrum Wiskunde \& Informatica (CWI), Amsterdam. M.M.\ and E.S.\ thank the CWI for its kind hospitality. The authors thank Johannes H\"owing for his comments.


\section{A generalized Gray-Scott-Klausmeier model}\label{s:gkgs}

For illustration of the subsequent considerations, let us consider the model \eqref{e:gkgs} for water-vegetation interaction in semi-arid landscapes. Here $A$ is roughly a measure of the rainfall. On the one hand, \eqref{e:gkgs} is (a rescaling of) the Klausmeier model for banded vegetation patterns on a sloped terrain from \cite{Klausmeier}, when removing the porous medium term $(w^2)_{xx}$. On the other hand, upon replacing $(w^2)_{xx}$ by $w_{xx}$ and setting $C=0$, \eqref{e:gkgs} is precisely the semilinear Gray-Scott model, which has been extensively studied in the past decades, see, e.g., \cite{CW09,DKZ97,MDK00} and the references therein. The relations between these different models in terms of periodic patterns have been studied in \cite{Sjors}. From an application point of view it is important to know in which patterned state these model systems may reside, and thus to establish well-posedness as well as existence, stability and instability  of patterns. 

In order to illustrate the straightforward applicability of the frameworks of  the following sections, we show well-posedness around travelling waves with first component bounded away from zero. We then consider homogeneous steady states and wavetrains, and derive the dispersion relations. These are illustrated by numerical computations of spectra when passing a Turing-Hopf bifurcation and a sideband instability.

\subsection{Well-posedness for perturbations of travelling waves}\label{s:gkgswell} To cast \eqref{e:gkgs} into the form \eqref{e:quasi} we set $\v = (w,v)$ and define the smooth nonlinearities $a:\R^2\to \R^{2\times 2}$ and $f:\R^2\to \R^{2}$ by
$$a(\v) = \left (\begin{array}{cc} 2w & 0 \\ 0 & D \end{array}\right),\qquad f(\v,\v_x) = \left (\begin{array}{c} C w_x + A(1-w) -wv^2 \\ -Bv +wv^2\end{array}\right).$$
Then \eqref{e:gkgs} is equivalent to
$$\v_t = (a(\v)\v_x)_x + f(\v,\v_x).$$
We see that $a(\v)$ is positive definite only for $w> 0$, and thus \eqref{e:gkgs} fails to be parabolic for $w\leq 0$. We therefore restrict to $w > 0$. From the quasi-positive structure of $f$ for $A>0$ and the smoothness of solutions given by the well-posedness, it readily follows that \eqref{e:gkgs} preserves $w>0$ on the maximal existence interval.

Assume that $\v_*(t,x) = \overline{\v}(x - ct)$ is a travelling wave solution of \eqref{e:gkgs} with profile 
$$\overline{\v} = (\bw,\bv) \in \BC^\infty(\R, \R^2)$$
satisfying $\bw\geq \delta > 0$, and speed $c\in \R$. Note that this includes homogeneous steady states. Denote the co-moving frame $x-ct$ again by $x$. As for \eqref{e:1234}, the evolution of perturbations $\v$ of $\overline{\v}$ under \eqref{e:gkgs} is governed by
\begin{equation}\label{e:1234-2}
\v_t = (a(\overline{\v}+\v) \v_x)_x  + ( a(\overline{\v}+\v) \overline{\v}_x)_x + c(\overline{\v}_x + \v_x) + f(\overline{\v}+\v, \overline{\v}_x + \v_x).
\end{equation}
Choose $\V$ as any open subset of $\X = H^2$, $\X = B_{2,p}^{2-2/p}$ with $p > 2$ sufficiently large or $\X = \BUC$ such that $\bw+w$ is positive and bounded away from zero for all $\overline{\v} = (\bw,\bv) \in \V$. This is possible in view of the Sobolev embeddings $H^2\subset \BUC$ and \eqref{e:Sob-Besov}. The Theorems \ref{t:quasiRDSpert}, \ref{t:quasiRDS-Rn} or \ref{t:lun-1D} apply and yield local well-posedness of \eqref{e:1234-2} in $\V$, respectively, in a sense as for the Theorems \ref{t:KPW-abstract} and \ref{t:lun-1D}. Solutions are in fact smooth in space and time (see Remark \ref{r:smooth}).  

The eigenvalue problem for the linearization of the right-hand side of \eqref{e:1234-2} in $\v = 0$ is for $\lambda\in \C$ given by 
\begin{equation}\label{e:ev-lin}
\begin{aligned}
\lambda w\,& =2\bw w_{xx}  + 4\bw_x w_x + 2\bw_{xx} w + (C+c)w_x-Aw-\bv^2w-2\bw\, \bv v, \\
\lambda v\, &=Dv_{xx} + c v_x -Bv+\bv ^2w+2\bw \,\bv v.
\end{aligned}
\end{equation}
By Proposition \ref{prop:spex-indie}, the spectrum of the linearization is independent of the above functional analytical frameworks. A brief account for the computation of the spectrum is given in $\S$\ref{s:specwtSob}, and we refer to \cite{BjornSurvey} for a survey. Nonlinear stability or instability of $\v_*$ can be deduced from the results in $\S$\ref{s:nl-stab} in some situations, as pointed out below.

\subsection{Homogeneous steady states}

\begin{figure}
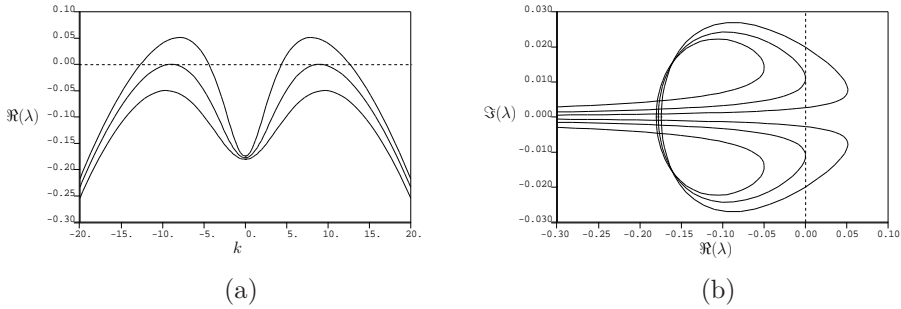

\begin{center}
\begin{tabular}{cc}
\scalebox{0.3}{\input{TuringHopf.tex}} & \scalebox{0.3}{\input{TuringHopf2.tex}} \\
(a) & (b)
\end{tabular}
\caption{Spectra of the homogeneous steady state $(w^+,v^+)$ of \eqref{e:gkgs} for $B=C=0.2$, $D=0.001$ before the Turing-Hopf instability, $A=0.63$ (stable), near to it, $A=0.53$, and after it, $A=0.43$ (unstable). (a) Real part of spectrum vs.\ linear wavenumber, (b)  Imaginary part of spectrum vs.\ real part.}
\label{f:turing}
\end{center}
\end{figure}

These are solutions $w(t,x)=w_*,v(t,x) =v_* \in \R$ to \eqref{e:gkgs} that are time and space independent, and thus solve the algebraic equations arising from vanishing space and time derivatives. We readily compute that the possibilities are $(w_0,v_0) = (1,0)$ and, in case $A \geq 4B^2$,
\begin{align*}
w_\pm = \frac 1 {2A} \left( A \mp \sqrt{A^2-4 A B^2}\right), \qquad  v_\pm = \frac 1 {2B} \left( A \pm \sqrt{A^2-4 A B^2}\right).
\end{align*}
The state $(w_0,v_0)$, with zero vegetation, represents the desert (even though there is non-zero `water'), while the equilibria $(w_{+},v_{+})$ and $(w_{-},v_{-})$ represent co-existing homogeneously vegetated states. At $A=A_{\rm{sn}}=4B^2$, the latter two collapse in a saddle-node bifurcation. The spectrum of the linearization in $(w_*,v_*)$ can be computed from the usual dispersion relation $d(\lambda,\kappa)= 0$, where 
\begin{align*}
 d(\lambda,\kappa)=\det\left(
 \begin{array}{cc}
 -2w_*\kappa^2+\rmi \kappa (C+c)-A-v_*^2 -\lambda & -2w_*v_* \\ 
 v_*^2 & -D\kappa^2+ \rmi \kappa c -B+2w_*v_* -\lambda
 \end{array} \right)
\end{align*}
is obtained from Fourier transform, see $\S$\ref{s:specwtSob}.

An origin of patterns is a (supercritical) Turing-Hopf bifurcation of the steady state $(w_+,v_+)$ that occurs as $A$ decreases from larger values, as shown  in \cite{Sjors}. It is in fact straightforward to study bifurcations of spatially periodic travelling waves as this only involves ODE analysis. As a sidenode on Turing-Hopf bifurcations, we mention that the dynamics of \eqref{e:gkgs} near onset is formally approximated by a complex Ginzburg-Landau equation (see \cite{Sjors}), but the rigorous justification has not been established for quasilinear problems, to our knowledge. 

In order to locate the Turing-Hopf bifurcation, we need to study the spectrum of the linearization in this state.  For illustration, in Figure~\ref{f:turing} we plot the spectrum obtained numerically (using \textsc{Auto} \cite{auto}) from the dispersion relation as the parameter $A$ passes through the aforementioned Turing-Hopf bifurcation. Since the spectrum is unstable after passing the Turing-Hopf instability (e.g. $A=0.43$ in Figure~\ref{f:turing}), the steady state is expected to be unstable under the nonlinear evolution. Indeed, this is the case thanks to Theorem~\ref{t:orb-instab}.

\subsection{Wavetrains}

\begin{figure}
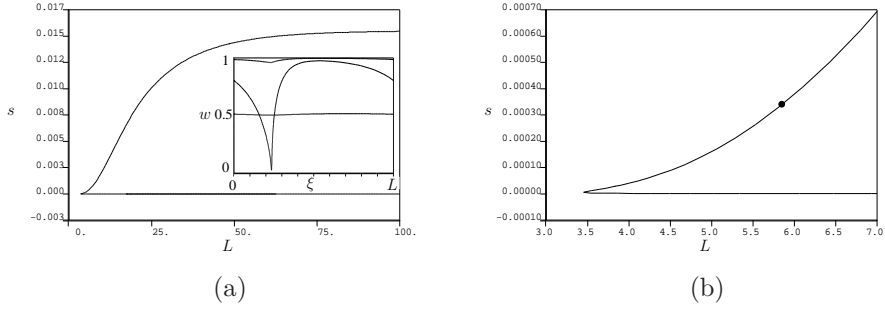

\begin{center}
\begin{tabular}{cc}
\scalebox{0.3}{\input{sample.tex}} & \scalebox{0.3}{\input{sampleZoom.tex}} \\
(a) & (b)
\end{tabular}
\caption{(a) Sample bifurcation diagram of wavetrains for $A=0.02$, $B=C=0.2$, $D=0.001$. At $L\approx 3.45$ a fold occurs, and both branches appear to terminate in a homoclinic bifurcation as $L\to\infty$. The inset shows profiles of solutions at the fold ($w\approx 0.5$) and near $L=80$ on upper and lower ($w\approx1$) branch. (b) Magnification of the bifurcation diagram with bullet marking the location of the sideband instability at $L\approx 5.98$. Solutions on the branch for increasing period are spectrally stable.}
\label{f:samplebif}
\end{center}
\end{figure}

The patterns emerging at the Turing-Hopf bifurcation are periodic wavetrains, which are solutions to \eqref{e:gkgs} of the form $$(w_*,v_*)(t,x) =  (\tilde w, \tilde v )(kx-\omega t),$$
with a $2\pi$-periodic profile $(\tilde w,\tilde v)$. Here $\omega$ is called the frequency and $k$ the wavenumber. As noted in \cite{Sjors}, the existence region of wavetrains to \eqref{e:gkgs} in parameter space extends far from the Turing-Hopf bifurcation and even beyond the saddle-node bifurcation $A=A_{\rm sn}$ of homogeneous equilibria with vegetation. In Figure~\ref{f:samplebif} we plot a branch of wavetrain solutions for $A<A_{\rm sn}$ that appears to terminate in another type of travelling waves: pulses, which are spatially homoclinic orbits.

In order to link to the formulations for travelling waves, let us cast wavetrains as equilibria $(w_*,v_*)(t,x) = (\bw,\bv)(x-ct)$ in the co-moving frame $x- ct$ with speed $c=\frac{\omega}{k}$. The eigenvalue problem of the linearization of \eqref{e:gkgs} in a wavetrain is then given by \eqref{e:ev-lin}, with coefficents of period $L=2\pi/k$ stemming from $(\bw,\bv)$.

\begin{figure}
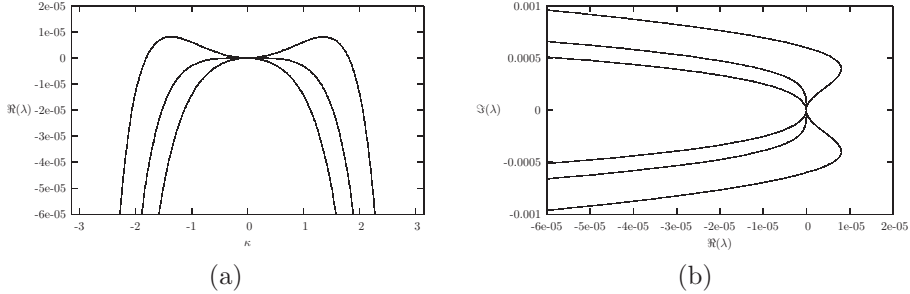

\begin{center}
\begin{tabular}{cc}
\scalebox{0.45}{\input{plot.tex}} & \scalebox{0.45}{\input{plotC.tex}} \\
(a) & (b)
\end{tabular}
\caption{Spectra of the wavetrains for $B=C=0.2$, $D=0.001$, $A=0.02$ before the sideband instability, $L=5.9$ (stable), near to it, $L\approx 5.98$, and after it, $L=6.1$ (unstable). (a) Real part vs.\ linear wavenumber, (b) Imaginary part  vs.\ real part. }
\label{f:turingWT}
\end{center}
\end{figure}

The approach via Fourier transform is less useful, because the linearization is not diagonal in Fourier space due to the $x$-dependent coefficients. As a substitute, one uses the Floquet-Bloch transform, which replaces the eigenvalue problem on $\R$ by a family of eigenvalue problems on the wavelength interval $[0,L]$ (see \S\ref{s:specwtSob}). Specifically, this can be cast as the family of boundary value problems for $\kappa\in[0,2\pi)$ given by \eqref{e:ev-lin} with $\partial_x$ replaced by $\partial_x+\rmi \kappa/L$ and $L$-periodic boundary conditions.

With a curve of spectrum of a wavetrain connected to the origin $\lambda=0$ (due to translation symmetry), a change in its curvature  is a typical destabilization upon parameter variation. This so-called \emph{sideband instability} is illustrated  in Figure~\ref{f:turingWT}, where we plot spectra of wavetrains in \eqref{e:gkgs} passing through a sideband instability as the wavelength $L$ changes. For these computations, we implemented the first order formulation of the dispersion relation numerically in \textsc{Auto} based on the algorithm from \cite{RSS}.

As for the homogeneous steady state, the wavetrains with unstable spectrum (e.g., $L=6.1$ in Figure~\ref{f:turingWT}) are expected to be (orbitally) unstable under the nonlinear evolution of \eqref{e:gkgs}, see Proposition \ref{prop:instab-wavetrain} and Theorem \ref{t:orb-instab}.


\section{Frameworks for well-posedness}\label{s:wellHil}
We formulate the abstract well-posedness results based on maximal regularity and present three concrete frameworks for quasilinear reaction-diffusion systems. In one space dimension we obtain well-posedness in $H^2$ and in $\BUC^2$, and in space dimensions less than or equal to three we have well-posedness in certain Besov spaces. More general problems and further settings are briefly discussed at the end of this section.

\subsection{Well-posedness based on maximal $L^p$-regularity} We formulate the results of \cite{KPW10, Pruess} for abstract quasilinear parabolic problems of the form
\begin{equation}\label{e:quasi-abs}
\partial_t u = A(u) u + F(u),\quad t> 0, \qquad u(0) = u_0,
\end{equation}
in a Hilbert space setting. Let $X_0,X_1$ be Hilbert spaces with $X_1$ continuously and densely embedded into $X_0$. Roughly speaking, $X_0$ is the base space for \eqref{e:quasi-abs} and $A(u(t))$ is an unbounded linear operator on $X_0$ with domain $X_1$. For a fixed number $p\in (1,\infty)$, consider the real interpolation space 
$$\X = (X_0,X_1)_{1-1/p,p}$$
between $X_0$ and $X_1$. This is the phase space in which the solution semiflow for \eqref{e:quasi-abs} acts. It is the analogue to the fractional power domains in the semilinear theory \cite{CD, Henry}. These two types of intermediate spaces differ, in general (with exceptions for $p=2$), but are closely related (see, e.g., \cite[Proposition 4.1.7]{Lunardi2}).  For the general properties of real interpolation spaces we refer to \cite{Bergh-Lof, Lunardi2, Triebel}. At this point we only note that $X_1\subset \X \subset X_0$ and that $\X$ is in general \emph{not} a Hilbert space. Explicit descriptions of $\X$ are available in our concrete settings below, e.g., $H^1 = (L^2, H^2)_{1/2,2}.$

Recall from \cite{EN, Lunardi1} that a densely defined operator $B$ on $X_0$ generates a strongly continuous analytic semigroup if and only if $\|\lambda (\lambda-B)^{-1} \|_{\mathscr L (X_0)}$ is uniformly bounded for $\lambda$  in a left open sector in $\mathbb C$. 

As a consequence of the results in \cite{KPW10, Pruess} we have

\begin{theorem} \label{t:KPW-abstract} Let $p\in (1,\infty)$ and $X_1 \subset \X \subset X_0$ be as above. Assume there is an open set $\mathcal V\subseteq \X$ such that
\begin{itemize}
   \item $F:\V \to X_0$ and $A: \V \to \mathscr{L}(X_1,X_0)$ are Lipschitz on bounded sets;
   \item for each $w_0\in \V$, the operator $A(w_0)$ with domain $X_1$ generates a strongly continuous analytic semigroup on $X_0$.
\end{itemize}
Then \eqref{e:quasi-abs} is locally well-posed in $\V$, with solutions in a strong $L^p$-sense. \end{theorem}

More precisely, the theorem yields solvability of \eqref{e:quasi-abs} as follows. For each initial value $u_0\in \V$ there is a maximal existence time $t^+(u_0) >0$ and a unique solution 
$u=u(\cdot;u_0)\in C([0,t^+(u_0)),\V)$ of \eqref{e:quasi-abs}, such that $u \in H^{1,p}(J,X_0)\cap L^p(J,X_1)$
for time intervals $J= (0,T)$ with $T < t^+(u_0)$. Here $H^{1,p}(J,X_0)$ denotes a vector-valued Sobolev space, which is defined as in the scalar case. Furthermore, $t^+(u_0)$ is finite only if either $\text{dist}(u(t;u_0), \partial \V)\to 0$ or $\|u(t;u_0)\|_{\X}\to \infty$ as $t\to t^+(u_0)$. The map $t^+:\V\to (0,\infty]$ is lower semicontinuous, and the local solution semiflow, $(t,u_0)\mapsto u(t;u_0)$, is continuous with values in $\V\subseteq \X$.  If $F$ and $A$ are smooth, then the semiflow enjoys smoothness properties as well. We demonstrate this in Proposition \ref{prop:time-T} in the appendix for a neighbourhood of a steady state.

Note that if $A(w_0)$ generates an analytic semigroup for $w_0\in \X$, then the Lipschitz property of $A$ as in the theorem combined with well-known perturbation results for semigroups (see  \cite[Proposition 2.4.2]{Lunardi1}) imply that this is true for any $A(\tilde w_0)$ with $\tilde w_0$ in a small neighbourhood of $w_0$. This gives a candidate for $\V$.

 To verify the assumptions in \cite[Section 2]{KPW10},  \cite[Theorem 3.1]{Pruess} and prove Theorem \ref{t:KPW-abstract} we only need to know that $-A(w_0)$ has  for each $w_0\in \V$ the property of maximal $L^p$-regularity on finite time intervals $J$. But in \emph{Hilbert spaces} this already follows from the assumed  generator property of $A(w_0)$. Indeed, by \cite[Theorems 3.3, 7.1]{Dore00} it suffices to consider the case $p=2$, $J=\R_+$ and that the semigroup generated by $A(w_0)$ is exponentially decaying. In this situation maximal $L^2$-regularity follows from \cite{DeSimon} (see also \cite[Theorem 1.6]{Pruess} for the short proof using Plancherel's theorem).

\subsubsection{One space dimension: well-posedness in $H^2$}\label{s:wellSob} For $u(t,x)\in \R^N$ we apply the  abstract result Theorem \ref{t:KPW-abstract} to the reaction-diffusion system 
\begin{equation}\label{e:RDS-1D}
u_t = (a(u) u_x)_x + f(u,u_x), \qquad t>0, \qquad x\in \R.
\end{equation}
To obtain a simple setting with familiar function spaces which is at the same time directly linked to $L^2$-spectral theory, we work with $X_0 = H^1= H^1(\R)^N$ as a base space. In one space dimension (and only there) this is possible since $H^1$ is even an algebra, i.e., $uw\in H^1$ and $\|uv\|_{H^1}\leq C\|u\|_{H^1}\|v\|_{H^1}$ for $u,w\in H^1$. 

We start with the case when the nonlinearities in \eqref{e:RDS-1D} are everywhere defined. We emphasize that $a$ does not have to be symmetric, and that $a, f$ may be less regular than actually stated.

\begin{theorem}\label{t:quasiRDSsmooth} Assume $a:\R^N \to \R^{N\times N}$ is $C^4$ such that $a(\zeta)\in \R^{N\times N}$ is positive definite for each $\zeta\in \R^N$, and that $f: \R^N\times \R^N \to \R^N$ is $C^3$ with $f(0,0) = 0$.

Then  \eqref{e:RDS-1D} is locally well-posed in the phase space $\X = H^2$. The solutions belong to $H^1(J, H^1(\R))\cap L^2(J, H^3(\R))\cap C(\overline{J}, H^2(\R))$
on time intervals $J = (0,T)$ away from the maximal existence time. \end{theorem}

\begin{proof}
We choose $X_0 = H^1$, $X_1 = H^3$ and $p = 2$. Then $\X = (H^1,H^3)_{1/2,2} = H^2$, see \cite[Remark 2.4.2/2]{Triebel}.  Define $A$ and $F$ by $A(u)v = (a(u)v_x)_x$ and $F(u) = f(u,u_x)$. Then $F:H^2\to H^1$ and $A:H^2\to \mathscr{L}(H^3, H^1)$ are Lipschitz on bounded sets by Lemma \ref{lem:superpos-H1}. For the generator property, let $w_0\in H^2$ be arbitrary. Denote by $A_{L^2}$ the realization of $A(w_0)$ on $L^2$, with domain $H^2$. Since $w_0, a(w_0) \in \BC^1$ by Sobolev's embedding $H^1\subset \BC$,  it follows from \cite[Corollary 9.5]{AHS94} that the operator $A_{L^2}$ generates an analytic $C_0$-semigroup on $L^2$. Next, let $A_{H^1}$ be the realization of $A(w_0)$ on $H^1$, i.e., the restriction of $A_{L^2}$ to $H^1$. Since $H^1 = (L^2,H^2)_{1/2,2}$ (see again \cite{Triebel}), it follows from \cite[Theorem 5.2.1]{Lunardi2} that $A_{H^1}$ with domain $D(A_{H^1}) = \{u\in H^2:A_{L^2}u \in H^1\}$ generates an analytic $C_0$-semigroup as well. Using the algebra property of $H^1$, it is elementary 
to check that $D(A_{H^1}) = H^3$ (see the proof of Lemma \ref{l:iso} in the appendix). Thus Theorem \ref{t:KPW-abstract} applies.
\end{proof}

\begin{Remark}\label{r:smooth}
Employing, e.g., \emph{Angenent's parameter trick} (see \cite[Theorem 5.1]{Pruess} and \cite{EPS}), one can show that for smooth nonlinearities the solutions of \eqref{e:RDS-1D} are smooth in space and time.
\end{Remark}

When investigating the stability of a non-localized travelling wave with respect to localized perturbations, one is lead to a variant of \eqref{e:RDS-1D} with $x$-dependent nonlinearities. Furthermore, in many situations the nonlinearities are not everywhere defined on $\R^N$, or the leading coefficient $a$ is positive definite only in a subset of $\R^N$. For instance, this is the case for the Gray-Scott-Klausmeier model \eqref{e:gkgs}, where the focus lies on perturbations of travelling wave solutions in the parabolic regime $w> 0$.

For a general formulation, let $\overline{u}\in \BC^2(\R,\R^N)$ be a steady state of \eqref{e:RDS-1D}, i.e.,
\begin{equation}\label{e:steady}
(a(\overline{u}) \overline{u}_x)_x + f(\overline{u}, \bu_x) = 0.
\end{equation}
Then $\overline{u} + u$ solves \eqref{e:RDS-1D} for a perturbation $u$ if and only if $u$ solves
\begin{equation}\label{e:RDS-1D-x}
u_t = (a(\overline{u}+u) u_x)_x + ( a(\overline{u}+u) \overline{u}_x)_x + f(\overline{u}+u,\bu_x + u_x).
\end{equation}
For this perturbative setting we have the following variant of Theorem \ref{t:quasiRDSsmooth}. Here and in the following, the image of $\overline{u}$ is meant to be the set $\{\overline{u}(x):x\in\R\}$.

\begin{theorem} \label{t:quasiRDSpert} Let $\bu \in \BC^2(\R,\R^N)$ satisfy \eqref{e:steady}, and let $U_1,U_2\subseteq \R^N$ be open neighbourhoods of the closure of the images of $\bu$ resp. $\bu_x$. Assume $a: U_1\to \R^{N\times N}$ is $C^4$ such that $a(\zeta)$ is positive definite for any $\zeta\in U_1$, and $f:U_1\times U_2\to \R^N$ is $C^3$.

Then there is an open neighbourhood $\mathcal V$ of the zero function in $H^2$ such that \eqref{e:RDS-1D-x} is locally well-posed in $\mathcal V$. If $U_1 = U_2 = \R^N$, then one can take $\mathcal V= H^2$.
\end{theorem}

\begin{proof} Let again $X_0 = H^1$, $X_1 = H^3$ and $p = 2$, such that $\X = H^2$.  Define
\begin{equation}\label{e:superpos-proof}
A(u)v = (a(\overline{u}+u) v_x)_x, \qquad F(u) =   ( a(\overline{u}+u) \overline{u}_x)_x + f(\overline{u}+u,\bu_x + u_x).
\end{equation}
Using $F(0) = 0$, Lemma \ref{lem:superpos-H1} yields $\V \subseteq H^2$ such that $F:\V\to H^1$ and $A:\V\to \mathscr L(H^3, H^1)$ are Lipschitz on bounded sets. If $\V$ is sufficiently small, then for each $w_0\in \mathcal V$ the leading coefficient $a(\overline{u}+w_0)$ of $A(w_0)$ is positive definite, uniformly in $x\in \R$. Thus as in the proof of  Theorem \ref{t:quasiRDSsmooth} it follows from \cite[Corollary 9.5]{AHS94} and an interpolation argument that $A(w_0)$ with domain $H^3$ has the required generator property on $H^1$ to apply Theorem \ref{t:KPW-abstract}.
\end{proof}

\subsubsection{Well-posedness in space dimensions $n\leq 3$} \label{s:well} For simplicity, on $\R^n$ we consider quasilinear reaction-diffusion-advection problems (using sum convention)
\begin{equation}\label{e:quasi-Rn}
 u_t =  \partial_i (a_{ij}(u) \partial_j u)  +  c_i \partial_i u + f(u), \qquad x\in \R^n.
\end{equation}
Here, essentially, $a_{ij}:\R^{N} \to \R^{N\times N}$, $c_i\in \R^{N\times N}$ for $i,j = 1,...,n$ and $f:\R^{N}\to \R^N$.
The approach of the previous subsection works in any dimension if one takes $X_0 = H^k(\R^n)$ with $k > \frac{n}{2}$ as a base space, since then $H^k$ is an algebra and the superposition operators are Lipschitz as before.

We present another functional analytic setting with $X_0=L^2$ as a base space, for which Theorem \ref{t:KPW-abstract}  applies to  \eqref{e:quasi-Rn} in space dimensions $n\leq 3$. The price one has to pay in the maximal $L^p$-regularity approach is that the phase space $\mathcal X = (L^2,H^2)_{1-1/p,p}$ becomes slightly more complicated to describe. It is the $N$-fold product $B_{2,p}^{s}$ of a \emph{Besov space} $B_{2,p}^{s}(\R^n)$, with $s >0$ and $p\in (1,\infty)$. For $s\notin \N$, it follows from \cite[Theorem 2.6.1]{Triebel2} that an equivalent norm for this space is given by 
$$\|u\|_{B_{2,p}^{s}}  = \|u\|_{H^{k}} + \sum_{|\alpha|\leq k} \Big ( \int_{|h|\leq 1} |h|^{-(s-k)p -n} \|D^\alpha u(\cdot + h)-D^\alpha u(\cdot)\|_{L^2}^p \,dh\Big)^{1/p},$$
where $k$ is the largest integer smaller than $s$. The Besov spaces are closely related to the more common Bessel-potential spaces $H^{s}$.  For any $\varepsilon > 0$ we have the dense inclusions $H^{s+\varepsilon} \subset B_{2,p}^{s} \subset H^{s-\varepsilon}$. However, $B_{2,p}^{s} = H^s$ if and only if $p=2$, and furthermore $B_{2,p}^{s}$ is a Hilbert space only for $p = 2$. Essential for the applications are the Sobolev embeddings
\begin{equation}\label{e:Sob-Besov}
B_{2,p}^s(\R^n) \subset \BC(\R^n)\quad \text{for }s > \frac{n}{2}, \quad B_{2,p}^s(\R^n)\subset L^q(\R^n)\quad  \text{for }s \geq \frac{n}{2} - \frac{n}{q}>0.
\end{equation}
These are a consequence of $B_{2,p}^{s} \subset H^{s-\varepsilon}$ and the corresponding embeddings for the $H$-spaces. For these and many more properties of $B$-spaces we refer to \cite{Triebel}.

As above we consider a perturbative setting. Analogous to \eqref{e:RDS-1D-x}, for perturbations $u$ of a steady state $\bu\in \BC^2(\R^n,\R^N)$ of \eqref{e:quasi-Rn}, one is lead to 
\begin{equation}\label{e:quasi-Rn-x}
\partial_t u = \partial_i (a_{ij}(\overline{u}+ u) \partial_j u)  +   \partial_i (a_{ij}(\overline{u}+ u) \partial_j \overline{u})  + c_i \partial_i (u +\overline{u})  + f(\overline{u}+u).
\end{equation}
Note that the following well-posedness result in particular applies to \eqref{e:quasi-Rn} when setting $\overline{u} = 0$ and assuming $f(0) = 0$. Again no symmetry properties of the diffusion coefficients $(a_{ij})$ are required.

\begin{theorem} \label{t:quasiRDS-Rn} Let $n= 1,2,3$. Let $\bu\in \BC^2(\R^n,\R^N)$ be a steady state of \eqref{e:quasi-Rn}, and let $U\subseteq \R^N$ be an open neighbourhood of the closure of its image. For all $i,j = 1,...,n$, assume that $c_i\in \R^{N\times N}$ is constant, that $a_{ij}:U \to \R^{N\times N}$ and $f: U\to \R^N$ are $C^2$, and that $a_{ij}(\zeta)$ is positive definite for any $\zeta \in U$. 

Then for all sufficiently large $p\in (2,\infty)$ there is an open neighbourhood $\mathcal V$ of the zero function in $B_{2,p}^{2-2/p} = B_{2,p}^{2-2/p}(\R^n)^N$ such that \eqref{e:quasi-Rn-x} is locally well-posed in $\mathcal V$. The solutions belong to $H^{1,p}(J, L^2)\cap L^p(J, H^2)\cap C(\overline{J}, \mathcal V)$ on time intervals $J$ away from the maximal existence time. If $U = \R^N$, then one can take $\mathcal V = B_{2,p}^{2-2/p}$.
\end{theorem}
\begin{proof} The choice $X_0 = L^2$ and $X_1 = H^2$ leads to $B_{2,p}^{2-2/p} =\X =  (X_0,X_1)_{1-1/p,p}$ for $p\in (1,\infty)$, see \cite[Remark 2.4.2/4]{Triebel}. Let $A(u)v =  \partial_i (a_{ij}(\overline{u}+ u) \partial_j v)$, and denote by  $F(u)$ the remaining terms on the right-hand side of \eqref{e:quasi-Rn-x}. The Lipschitz properties of $A$ and $F$ on a neighbourhood $\V$ of zero follow from Lemma \ref{lem:superpos-n}. For $w_0\in \mathcal V$ the operator $A(w_0)$ is elliptic, the coefficients are bounded and the leading coeffient is uniformly H\"older continuous, since $B_{2,p}^{2-2/p}$ even embeds into $\BC^\sigma$ for some $\sigma > 0$ if $n\leq 3$ and $p$ is large, see \cite[Theorem 2.8.1]{Triebel}. Now the generator 
property on $L^2$ follows again from \cite[Corollary 9.5]{AHS94}.
\end{proof}

\subsection{Well-posedness based on maximal H\"older regularity} We formulate the well-posedness result of \cite[Chapter 8]{Lunardi1} for abstract quasilinear parabolic problems
\begin{equation}\label{e:quasi-abs-2}
\partial_t u = A(u) u + F(u),\quad t> 0, \qquad u(0) = u_0.
\end{equation}
The approach of \cite{Lunardi1} is based on maximal H\"older regularity (see also \cite[Chapter III.2]{AmannBook} for the general linear theory). It also covers fully nonlinear problems and does not take into account the quasilinear structure of \eqref{e:quasi-abs-2}. It has the big advantage to be applicable in arbitrary Banach spaces $X_0$, while in applications maximal $L^p$-regularity is usually restricted to reflexive Banach spaces, excluding spaces of continuous functions. Moreover, the phase space equals the domain of the linearized operator, which is usually easier to describe than an interpolation space.

The following well-posedness result for \eqref{e:quasi-abs-2} is a consequence of \cite[Theorem 8.1.1, Proposition 8.2.3, Corollary 8.3.3]{Lunardi1}. 

\begin{theorem} \label{t:lun-abstract} Let $X_0, X_1$ be arbitary Banach spaces such that $X_1$ is continuously and densely embedded in $X_0$. Let $\V \subseteq \X:= X_1$ be open, define $\mathcal F(u) = A(u)u + F(u)$ and suppose that
\begin{itemize}
\item $\mathcal F\in C^1(\V,X_0)$ with locally Lipschitz derivative;
\item for each $w_0\in \V$, the operator $\mathcal F'(w_0)$ with domain $X_1$ generates a strongly continuous analytic semigroup on $X_0$, and $\|u\|_{X_0} + \|\mathcal F'(w_0)u\|_{X_0}$ defines an equivalent norm on $X_1$.
\end{itemize}
Then \eqref{e:quasi-abs-2} is locally well-posed in $\V$, and solutions are classical in time.
\end{theorem}

As already mentioned, the phase space $\X$ is now a subset of $X_1$ and not of an intermediate space between $X_0$ and $X_1$. Well-posedness is similar as for Theorem \ref{t:KPW-abstract}.  The maximal existence time is lower semicontinuous and the solution semiflow is continuous with values in $\V$. For each $\alpha\in (0,1)$ and an initial value $u_0\in \V$, one obtains a unique maximal solution $u$ of \eqref{e:quasi-abs-2} such that $u\in \BUC^{1+\alpha}_\alpha([0,T],X_0)\cap \BUC_\alpha^\alpha([0,T], X_1)$ for $T<t^+(u_0)$. Here $\BUC_\alpha^\alpha$ is a weighted H\"older space, see \cite[Chapter III.2]{AmannBook} and \cite[Example 3]{PSZ09-2}. (It is slightly confusing that these spaces differ from the ones in \cite{Lunardi1} denoted by $C_\alpha^\alpha$, but $\BUC_\alpha^\alpha$ is indeed the regularity obtained in \cite[Theorem 8.1.1]{Lunardi1}).

Theorem \ref{t:lun-abstract} applies to \eqref{e:RDS-1D}, \eqref{e:RDS-1D-x} and \eqref{e:quasi-Rn-x} under similar assumptions as in the Theorems \ref{t:quasiRDSsmooth}, \ref{t:quasiRDSpert} and \ref{t:quasiRDS-Rn}, with different phase spaces. In particular, instead of a Besov space one obtains $H^2$ as a phase space in the setting of Theorem \ref{t:quasiRDS-Rn}. We do not formulate the precise results and rather consider a setting for reaction-diffusion systems which is not covered by the approach of Theorem \ref{t:KPW-abstract}.

\subsubsection{One space dimension: well-posedness in $\BUC^2$} We reconsider the case of one space dimension, i.e., for $u(t,x)\in \R^N$ the problem
\begin{equation}\label{e:1D-rev}
u_t = (a(u) u_x)_x + f(u,u_x), \qquad t>0, \qquad x\in \R.
\end{equation}
We present a setting in which non-localized perturbations of steady states can be treated. For $k\in \N_0$, denote by $\BUC^k = \BUC^k(\R,\R^N)$ the Banach space of bounded uniformly continuous functions, endowed with the usual $C^k$-norm.
It is shown in \cite{Lunardi1} that a scalar second order elliptic operator on $\BUC = \BUC^0$ behaves well and generates an analytic semigroup. This is the main ingredient to apply Theorem \ref{t:lun-abstract} as follows. The triangular structure of $a$ is assumed for simplicity.

\begin{theorem} \label{t:lun-1D} Let  $\overline{u}\in \BUC^2(\R, \R^{N})$ be a steady state of \eqref{e:1D-rev} and let $U_1,U_2\subseteq \R^N$ be open neighbourhoods of the closure of image of $\bu$ resp. $\bu_x$. Assume  $a:U_1\to \R^{N\times N}$ and $f:U_1\times U_2\to \R^N$ are $C^2$, such that
\begin{itemize}
\item for each $\zeta\in U_1$ the matrix $a(\zeta)$ is triangular, and the diagonal entries of $a$ are positive and bounded away from zero uniformly.
\end{itemize}
Then there is an open neighbourhood $\V$ of $\overline{u}$ in $\BUC^2$ such that \eqref{e:1D-rev} is locally well-posed in $\V$. One can take $\V = \BUC^2$ if $U_1= U_2 = \R^N$.
\end{theorem}
\begin{proof} Choose an open set $V\subset \R^N$ that contains the image of $\overline{u}$ and satisfies $\overline{V}\subset U$. Define $\V$ as the set of all $w_0\in \BUC^2$ with image contained in $V$. Then $\mathcal F(u) = (a(u) u_x)_x +  f(u,u_x)$ defines a superposition operator $\mathcal F:\V\to \BUC$. It is straightforward to check that $\mathcal F\in C^1(\V, \BUC)$. At $w_0\in \V$ we have 
$$\mathcal F'(w_0)v = (a(w_0)v_x)_x + (a'(w_0)[(w_{0})_x,v])_x +c v_x + f'(w_0)v,\qquad v\in \BUC^2,$$
 and $\mathcal F':\V\to \mathscr L(\BUC^2, \BUC)$ is locally Lipschitz. For the generator property, let $w_0\in \V$ be given. By \cite[Corollary 3.1.9]{Lunardi1}, each of the scalar-valued operators $v \mapsto a_{ii}(w_0)v_{xx}$ with domain $\BUC^2$ generates an analytic $C_0$-semigroup on $\BUC$, where $a_{ii}$ are for $i = 1,...,N$ the diagonal entries of $a$. Using the matrix generator result \cite[Corollary 3.3]{Nag89} and the triangular structure of $a$, we conclude that the principle part $v\mapsto a(w_0)v_{xx}$ of $\mathcal F'(w_0)$ is a generator on $\BUC(\R,\R^N)$, with domain $\BUC^2(\R, \R^N)$. The remaining lower order terms preserve this property. The equivalence of the graph norm of $\mathcal F'(w_0)$ and the $C^2$-norm follows from the boundedness of the coefficients and the open mapping theorem. Therefore Theorem \ref{t:lun-abstract} applies to \eqref{e:1D-rev}.
\end{proof}
 
\subsection{More general problems and other frameworks}\label{s:outlook}
The above results also hold for smooth $x$-dependent nonlinearities, provided the principal term $a$ is positive definite uniformly in $x$. Also non-autonomous and nonlocal problems can be treated, see \cite{Amann2, Amann3, Lunardi1, Pruess}. Only the mapping properties of the superposition operators and the generator properties of the linearization are relevant. Both frameworks cover general quasilinear systems in any dimension if one works with $X_0= L^q$ for large $q$ as a base space, since then the superposition operators are well-defined by Sobolev embeddings. Theorem \ref{t:lun-abstract} also allows to work in spaces of H\"older continuous functions, $L^\infty$ or subspaces of $\BUC$ like $C_0$ or $C(\overline{\R})$, based on the analytic generator results  of \cite{Lunardi1} and \cite[Section VI.4]{EN}.

A framework with spatial weights might also be of interest, for instance, to force some decay of solutions \cite{Zumbrun} or to treat singular terms \cite{Mey}. Here in particular weights with exponential growth are straightforward to treat, as the generator results can be obtained from the unweighted case by a simple similarity transformation.

Besides the above approaches based on maximal $L^p$- and H\"older regularity there is a similar abstract approach based on continuous regularity \cite{Ang, CS}. Completely different frameworks for problems  in weaker settings on bounded domains with boundary conditions are presented in \cite{Amann, GR}. They should also be applicable to problems on $\R^n$. Finally, the poineering work of \cite{LSU} should be mentioned. For a comprehensive overview of possible settings for quasilinear parabolic problems we refer to \cite{Amann2}.


\section{Stability and spectra of travelling waves}\label{s:spec} While travelling waves also occur in higher space dimensions, we restrict here to $x\in\R$. 

Throughout, let $u_* (t,x) = \bu(x-ct)$ be a travelling wave solution of 
$$u_t = (a(u)u_x)_x + f(u,u_x), \qquad x\in \R,$$
with speed $c\in \R$ and profile $\bu \in \BC^\infty(\R,\R^N)$ solving the ordinary differential equation \eqref{e:steady}.
We assume that $a,f$ are $C^\infty$ and that $a$ is uniformly positively definite in a vicinity of the image of $\bu$. Suitable finite regularity of $\bu, a,f$ suffices for each of the following results and we assume infinite smoothness only for the sake of a simple exposition.  We further assume that $\bu$ is constant or periodic at infinity and that the asymptotic states are approached exponentially.  A travelling wave is called a \emph{pulse} or a \emph{front} if the asymptotic states are equal or different homogeneous equilibria, respectively. A \emph{wavetrain} is a periodic travelling wave, and we refer to travelling waves with at least one periodic asymptotic state as \emph{generalized} fronts or pulses.

\subsection{Stability in a perturbative setting} The evolution of perturbations $u$ of $u_*$ is governed by
\begin{equation}\label{e:1234}
u_t = (a(\overline{u}+u) u_x)_x + ( a(\overline{u}+u) \overline{u}_x)_x + c(\bu_x + u_x) + f(\overline{u}+u, \bu_x + u_x),
\end{equation}
where the co-moving frame $x-ct$ is again denoted by $x$. By translation invariance of the underlying equation, stability must be considered with respect to the family of translates
$$S= \{\bu(\cdot+\tau)- \bu:\tau\in \R\}.$$
The Theorems \ref{t:quasiRDSpert}, \ref{t:quasiRDS-Rn} and \ref{t:lun-1D} guarantee local well-posedness of \eqref{e:1234} for initial data from $\mathcal X= H^2$, $\X = B_{2,p}^{2-2/p}$ or $\X = \BUC^2$ sufficiently close to $S$ (note that in Theorem  \ref{t:quasiRDS-Rn} it is actually assumed that $f$ is independent of $u_x$). Even though $H^2 \subset B_{2,p}^{2-2/p}$ we distinguish between these cases, because of the different corresponding base spaces $H^1$ and $L^2$, and to highlight that a pure Sobolev space setting suffices for \eqref{e:1234}. For $\X = \BUC^2$, or in case of a pulse, one could equivalently consider \eqref{e:1234} with $\bu$ replaced by zero, in a neighbourhood of $\{\bu(\cdot+\tau):\tau\in \R\}$.

If $u_*$ is a pulse or a front, then $S$ is in each setting a family of equilibria of \eqref{e:1234}. 

\begin{definition} \label{def:stab} A pulse or front solution $u_*$  is called \emph{orbitally stable}, if for each initial value $u_0 \in \X$ sufficiently close to $S$ the corresponding solution of \eqref{e:1234} exists and stays as close as prescribed to $S$ for all positive times.  $u_*$ is called \emph{orbitally stable with asymptotic phase}, if it is orbitally stable and if for each $u_0\in \X$ sufficiently close to $S$ there is $\tau_\infty$ such the corresponding solution of \eqref{e:1234} converges to $\bu(\cdot+\tau_\infty)- \bu$ as $t\to \infty$. $u_*$ is \emph{orbitally unstable} if it is not orbitally stable.
\end{definition}

For a wavetrain, translates of the profile cannot be realized by localized perturbations. Thus only for $\X = \BUC^2$ orbital stability as above can be considered. For localized perturbations, i.e., $\mathcal X= H^2$ or $\X = B_{2,p}^{2-2/p}$, stability of a wavetrain is understood with respect to stability of the zero solution of \eqref{e:1234}.

\subsection{The spectrum of the linearization} The linearization $\L$ of the right-hand side of \eqref{e:1234} in $u = 0$ is 
\begin{align}\label{e:Linspecial}
\L \varphi=\alpha  \varphi_{xx} + \beta \varphi_x + \gamma\varphi,
\end{align}
with smooth coefficients $\alpha(x),\beta(x),\gamma(x)\in \R^{N\times N}$ given by
$$\alpha = a(\bu), \qquad \beta = a'(\bu)[\bu_x, \cdot] + a'(\bu)[\cdot,\bu_x]+c + \partial_2f(\bu,\bu_x),$$
$$ \gamma = a''(\bu)[\bu_x,\cdot,\bu_x] + a'(\bu)[\cdot,\bu_{xx}] + \partial_1 f(\overline{u},\bu_x).$$
Depending on the chosen well-posedness framework, the operator $\L$ is considered on $X_0=H^1$, $L^2$ or $\BUC$, with domain $H^3$, $H^2$ or $\BUC^2$, where we write $\L_{X_0}$ for a realization.

If translations of the profile can be realized by perturbations in $\X$, i.e., for pulses and fronts in any setting and for  wavetrains for $\X = \BUC^2$, then, by translation invariance, $\lambda = 0$ is an eigenvalue of $\L$ with eigenfunction $\bu_x$.

As for the approach surveyed in \cite{BjornSurvey}, we distinguish between the \emph{point spectrum}, i.e., $\lambda\in \spec\,\L_{X_0}$ such that $\L_{X_0}-\lambda$ is a Fredholm operator of index zero, and the complementary \emph{essential spectrum}. We will see that point and essential spectrum are independent of the chosen framework and that the familiar spectral theory for ordinary differential operators based on exponential dichotomies, as described in \cite{BjornSurvey}, applies to $\L$.

Usually the \emph{set of eigenvalues} is called point spectrum. Note that with the above definition eigenvalues can be contained in the essential spectrum. Eigenvalues are not independent of the setting: realized on $\BUC$, zero is an eigenvalue for $\phi\mapsto \phi' - \text{i}\phi$, but it is not an eigenvalue for its realization on $L^2$ and $H^1$. Of course this does not contradict Proposition \ref{prop:spex-indie} on kernel dimensions below since the operator is not Fredholm.

Since it is assumed that $a$ is positive definite in a neighbourhood of the image of $\bu$, the multiplication by $\alpha^{-1}$ is an isomorphism in each setting. Thus the invertibility and Fredholm properties of $\L-\lambda$ are the same as for
$$\tilde \L(\lambda)= \alpha^{-1} (\L - \lambda) = \partial_{xx} + \alpha^{-1}\beta \partial_x + \alpha^{-1}(\gamma-\lambda),$$
which has constant leading order coefficients.  As before we write $\tilde \L_{X_0}(\lambda)$ for a realization of $\tilde \L(\lambda)$.  The key to the spectral properties of $\tilde \L(\lambda)$ is the corresponding first order operator 
$$\tilde \calT(\lambda) = \partial_x - A(\cdot,\lambda), \qquad A(x,\lambda) = \left ( \begin{array}{cc} 0 &-1\\ \alpha^{-1}(x)(\gamma(x)-\lambda) & \alpha^{-1}(x) \beta(x) \end{array} \right),$$ 
which is obtained from rewriting $\tilde \L(\lambda) = 0$ into a first order ODE. Hence $A(x,\lambda)$ is a $(2N\times 2N)$-matrix. We write $\tilde \calT_{L^2}(\lambda)$ and $\tilde \calT_{BUC}(\lambda)$ for the realization of $\tilde \calT(\lambda)$ on $L^2(\R, \C^{2N})$ and $\BUC(\R, \C^{2N})$ with natural domains, respectively.

The following result is rather folklore, but does not seem to be explicitly stated in the literature. The equality of spectra for realizations on $\Lspace^p$, $1\leq p <\infty$ and the space $C_0$ of continuous functions vanishing at infinity follows from \cite[Corollary 4.6]{RS}. For the more general theory of dichotomies and spectral mapping results on these spaces we refer to the monograph \cite{CL99}.

\begin{proposition} \label{prop:spex-indie} The following assertions are true, where $\lambda\in \C$.
\begin{itemize}
 \item The spectrum, the point spectrum and the essential spectrum of $\L_{H^1}$, $\L_{L^2}$ and $\L_{\BUC}$ all coincide, respectively. 
 \item The operator $\L_{L^2}-\lambda$ is invertible if and only if $\tilde \calT_{L^2}(\lambda)$ is invertible. 
 \item The operator $\L_{L^2}-\lambda$ is Fredholm if and only if $\tilde \calT_{L^2}(\lambda)$ is Fredholm. In this case the Fredholm indices coincide, as well as the dimension of the kernels.
\end{itemize}
\end{proposition}

\begin{proof}  Lemma \ref{l:iso} provides an isomorphism $T$ from $H^1$ to $L^2$ and from $H^3$ to $H^2$ such that $\L_{H^1} = T^{-1} \L_{L^2} T$. Thus $\L_{H^1}-\lambda$ and $\L_{L^2}-\lambda$ have for each $\lambda\in \C$ the same invertibility and Fredholm properties. It remains to compare $\L_{L^2}-\lambda$ and $\L_{\BUC}-\lambda$. Since $\alpha$ is boundedly invertible, these operators have the same invertibility and Fredholm properties as $\tilde \L_{L^2}(\lambda)$ and $\tilde\L_{\BUC}(\lambda)$, respectively. It follows from \cite[Theorem~A.1]{SanSchIndices} that their Fredholm properties are the same as those of $\tilde \calT_{L^2}(\lambda)$ and $\tilde \calT_{\BUC}(\lambda)$, respectively. It is further clear that the dimensions of the kernels coincide in both settings. Now in \cite[Theorem 1.2]{BG} it is shown that the Fredholm properties of $\tilde \calT_{L^2}(\lambda)$ are characterized by exponential dichotomies of the ODE $v' = A(\cdot,\lambda)v$ on both half-lines, and that in this case the 
dimension of the kernel of $\tilde \calT_{L^2}(\lambda)$ only depends on the image of the dichotomies. This characterization is also true for $\tilde \calT_{\BUC}(\lambda)$ with the same formula for the dimension of the kernel, see \cite[Lemma 4.2]{Pa84} and \cite{Pa88}. Hence the invertibility and Fredholm properties of $\tilde \calT_{L^2}(\lambda)$ and $\tilde \calT_{\BUC}(\lambda)$ coincide, and if the operators are Fredholm, then the dimensions of the kernels coincide. This carries over to $\L_{L^2}-\lambda$ and $\L_{\BUC}-\lambda$ by the above considerations and shows the assertions. \end{proof}

We finally remark that also for the realization of $\tilde \calT(\lambda)$ on $L^q$ with any $1<q<\infty$ the Fredholm properties are characterized by exponential dichotomies (see \cite[p. 94]{BG}). Together with the arguments for \cite[Theorem~A.1]{SanSchIndices}, an appropriate generalization of Lemma \ref{l:iso} and interpolation. This shows that the spectrum of $\L$ is independent of its realization on any of the spaces $H^{s,q}$ and $B^s_{q,r}$, where $s\geq 0$ and $1\leq r\leq \infty$.

\subsection{Computation of the spectrum}\label{s:specwtSob} The invertibility and Fredholm properties of $\tilde \calT(\lambda)$, and thus the characterization of point and essential spectrum of $\L$, are described in terms of exponential dichotomies in \cite[Section 3.4]{BjornSurvey}. This is independent of the variable leading order coefficients of $\L$ due to its quasilinear origin, and thus the same as for semilinear reaction-diffusion systems. We briefly describe the main points for each type of wave.

For a \emph{homogeneous steady state} the point spectrum of the constant coefficient operator $\L$ is empty. Since the Fourier transform is an isomorphism on $L^2$, the (essential) spectrum can be determined by transforming $\L$ to 
$$\widehat \L(\kappa) = - \alpha\kappa^2 + \text{i}\beta\kappa + \gamma \in \C^{N\times N}, \qquad \kappa\in \R.$$ Now we have $\lambda\in \spec\,\L$ if and only if 
$$d(\lambda,\kappa) := \det (\widehat \L(\kappa) -\lambda) = \det (A(\lambda)-\text{i}\kappa) = 0$$
for some $\kappa$, which is called the dispersion relation for $\L$. The latter also means that $A(\lambda)$ is a non-hyperbolic matrix. Thus here it is straightforward to determine the spectrum, at least for $N$ not too large. 

For \emph{pulses and fronts}, replacing the variable coefficients of $\L$ by their values at $\pm\infty$ leads to constant coefficient operators $\L^\pm$ whose spectrum is determined as just described. For pulses the essential spectrum of $\L$ already coincides with $\spec\,\L^\pm$. For fronts, $\spec\,\L^\pm$ equals the boundary of the essential spectrum of $\L$, which is usually already sufficient to know for stability issues. This is related to the fact that the replacement by the values at infinity is a relatively compact perturbation of $\L$, which leaves Fredholm properties invariant (see \cite[Theorem IV.5.26]{Kato2}). The point spectrum of a pulse or a front is determined by detecting intersections of the stable and unstable subspaces of $v' = A(\cdot,\lambda)v$. Here the \emph{Evans function}  \cite{Evans,AGJ} is a powerful tool and we refer to the survey \cite[Section 4]{BjornSurvey} and the references therein.

For a \emph{wavetrain}, i.e., when $\bu$ is periodic with wavelength (period) $L > 0$, the coefficients of $\L$ are periodic. The point spectrum is empty. Instead of the Fourier transform, here the Floquet-Bloch transform applies and yields (see \cite[Theorem A.4]{Mielke97}, also for higher space dimensions)

\begin{align}\label{e:specBloch}
\spec\, \L = \overline{\cup_{\kappa\in[0,2\pi/L)}\spec\, \B(\kappa)}.
\end{align}
For $\kappa \in[0,2\pi/L)$ the operator $\B(\kappa):\Hspace^2_\per(0,L)\subset \Lspace_{2,\per}(0,L) \to  \Lspace_{2,\per}(0,L)$ with periodic boundary conditions is given by
\[
\B(\kappa)U=\rme^{-\rmi\kappa x}\L[\rme^{\rmi\kappa x}U] = \widehat\L(\rmi\kappa + \partial_x)U,
\]
where $\hat\L(\cdot)$ is the formal operator symbol of $\L$. Since $\spec\,\B(\kappa)$ only consists of eigenvalues, its spectrum is fully determined by the solvability of the family of boundary value problems
\[
\hat\L(\rmi\kappa+\partial_x)U=\lambda U\,,\qquad U(0) = U(L).
\]
In fact, also multiplicity of eigenvalues is determined via Jordan chains as in \cite{AGJ,BjornSurvey}. Notably, the spectrum again comes in curves; now an infinite countable union since the eigenvalue problem for each $\kappa$ still concerns an unbounded operator (rather than a matrix in case of a homogeneous steady state).  

Via $V=\rme^{\rmi\kappa x}U$, the boundary value problem formulation is equivalent to
\[
\hat\L(\partial_x)V=\lambda V\,,\qquad V(0) = \rme^{\rmi\kappa L}V(L).
\]
By Floquet theory, this precisely means that the period map $\Pi(\lambda)$ of the evolution operator for the ODE $\hat\L(\partial_x)U=\lambda U$ possesses an eigenvalue (a Floquet multiplier) $\rme^{\rmi\kappa L}$. Hence, also here a (linear) dispersion relation can be defined by
\[
d(\lambda, \kappa) = \det\left(\Pi(\lambda) - \rme^{\rmi\kappa L}\right)=0,
\]
which precisely characterizes the spectrum. An important difference to the case of homogeneous steady states is that $\lambda=0$ always lies in the essential spectrum: $x$-independent coefficients of \eqref{e:1234} yield a trivial zero Floquet exponent, which implies that $d(0,0)=0$. Indeed, $\B(0)\bu_x=0$ in this translation symmetric case. 

Finally, in case of a \emph{generalized wave train}, the boundary of the essential spectrum of $\L$ is as above obtained by replacing the coefficients of $\L$ with its periodic limits at $\pm\infty$, and considering the dispersion relation. The  point spectrum is also given by an Evans-function, see \cite[Section 4]{SSdefect} (here also the more general case of time periodic solutions, so-called defects, is treated).

\section{Nonlinear stability and instability} \label{s:nl-stab}  For the nonlinearities $a,f$ and a travelling wave solution $u_*(t,x) = \bu(x-ct)$ of \eqref{e:quasi} we make the same assumptions as in the previous section. We consider \eqref{e:1234} 
\begin{equation*}
u_t = (a(\overline{u}+u) u_x)_x  + ( a(\overline{u}+u) \overline{u}_x)_x + c(\bu_x + u_x) + f(\overline{u}+u, \bu_x + u_x)
\end{equation*}
in any of the well-posedness settings in a neighbourhood of $S  = \{\bu(\cdot+\tau)-\bu:\tau\in \R\}$.

\subsection{Stability of pulses and fronts} Recall the precise notion of orbital stability from Definition \ref{def:stab}. An application of \cite{PSZ09, PSZ09-2} gives the following conditional result. For more information on semisimple eigenvalues in Banach spaces we refer to \cite[Appendix A.2]{Lunardi1}.

\begin{proposition} \label{prop:stability} Let $\bu$ have constant  asymptotic states. Assume $\lambda = 0$ is a semisimple 
eigenvalue of $\mathcal L$ with eigenfunction $\bu'$, i.e., $\emph{\textrm{ker}}\, \L = \mathrm{span}\{\bu'\}$ and $X_0=\emph{\textrm{ker}}\, \L\oplus \mathrm{im}\,\L$. Assume further that the remaining part of $\spec\, \mathcal L$ is strictly contained in $\{\emph{\text{Re}}\,\lambda < 0\}$. Then the travelling wave $u_*$ is orbitally stable with asymptotic phase, and limit translates $u(\cdot+\tau_\infty)$ are approached exponentially.
\end{proposition}
\begin{proof} By translation invariance it suffices to consider $S$ in a neighbourhood of $\tau = 0$. The framework of Theorem \ref{t:KPW-abstract} is the one of \cite[Theorem 2.1]{PSZ09}, provided that, in addition, $A$ and $F$ belong to $C^1$, which is guaranteed by the assumption on $a$ and $f$. The setting of Theorem  \ref{t:lun-abstract} is the one of \cite[Example 3]{PSZ09-2}. To apply  \cite[Theorem 2.1]{PSZ09} and \cite[Theorem 3.1]{PSZ09-2} it remains to verify that zero is normally stable, in the sense of \cite{PSZ09, PSZ09-2}. We have that $S$ is a one-dimensional $C^1$-manifold, with tangent space at $\tau = 0$ spanned by $\bu'$. By assumption, the tangent space coincides with the kernel of $\mathcal L$ and zero is a semisimple eigenvalue. Hence normal stability follows.  \end{proof}

For a quasilinear variant of the Huxley equation, the above conditions have been verified in \cite[Section 5]{PSZ09} by elementary arguments.

\subsection{Instability of generalized pulses and fronts under localized perturbations} For localized perturbations, i.e., for $\X = H^2$ or $B_{2,p}^{2-2/p}$, a  generalized pulse or front $u_*$ is nonlinearly stable or unstable if the zero solution of \eqref{e:1234} is stable or unstable, as a single equilibrium in the sense of Lyapunov. Nonlinear stability is a delicate issue (see the discussion in the introduction). In case of an unstable spectral value we have the following.

\begin{proposition} \label{prop:instab-wavetrain} If $\bu$ has a periodic asymptotic state and $\spec\,\mathcal L\cap \{\emph{\text{Re}}\,\lambda > 0\} \neq \emptyset$, then the generalized  front or pulse  $u_*$ is nonlinearly unstable with respect to localized perturbations from $\X = H^2$ or $\X = B_{2,p}^{2-2/p}$.
\end{proposition}
\begin{proof} The Lemmas \ref{lem:superpos-H1} and \ref{lem:superpos-n} together with Proposition \ref{prop:time-T} imply that the time-one solution map $\Phi_1$ for \eqref{e:1234} obtained in Theorems \ref{t:quasiRDSpert} and \ref{t:quasiRDS-Rn} from Theorem \ref{t:KPW-abstract}  is $C^2$ around zero, with $\Phi_1'(0) = e^{\mathcal L}\in \mathscr L(\X)$. Considered on $\mathscr L(X_0)$, this operator has spectral radius larger than one by  \cite[Corollary 2.3.7]{Lunardi1}. Using $\mathcal L-\omega$ with sufficiently large $\omega > 0$ as a conjugate, this property carries over to $e^{\mathcal L}$ considered on $\mathscr L(X_1)$. Now it follows from interpolation that  the realization of $e^{\mathcal L}$ on $\mathscr L(\X)$ has spectral radius greater than one. Thus the zero solution of \eqref{e:1234} is unstable by \cite[Theorem 5.1.5]{Henry}.
\end{proof}

\subsection{Orbital instability} Without assuming a spectral gap or the existence of an unstable eigenvalue we show that an unstable spectrum implies orbital instability.

\begin{theorem} \label{t:orb-instab} The following assertions are true.
\begin{itemize}
\item Let $\bu$ have constant asymptotic states. Assume $\spec \,\mathcal L \cap \{\emph{\text{Re}}\, \lambda >0\} \neq \emptyset$. Then $u_*$ is orbitally unstable with respect to localized and non-localized perturbations from $\X = H^2$, $B_{2,p}^{2-2/p}$ or $X = \BUC^2$.

\item  Let $\bu$ have a periodic asymptotic state. Assume $\spec \,\mathcal L \cap \{\emph{\text{Re}}\, \lambda >0\} \neq \emptyset$. Then $u_*$ is orbitally unstable with respect to non-localized perturbations from $\X = \BUC^2$.
\end{itemize}
\end{theorem}

This result is a direct consequence of the general orbital instability result Theorem \ref{t:instab-general} below for manifolds of equilibria: $\bu'\in X_1$ in the settings under consideration and $\mathcal L \bu' = 0$ by the exponential convergence of $\bu'$ at infinity and translation invariance of the equation.

The following lemma and its proof are generalizations of  \cite[Theorem 5.1.5]{Henry}. Similar to that result, the proof establishes that perturbations of suitable approximate unstable eigenfunctions deviate from the manifold of equilibria.

\begin{lemma} \label{l:Henry-lem} Let $X$ be a real Banach space, let $\V\subseteq X$ be an open neighbourhood of zero and let $\mathcal E \subset \V$ be an $m$-dimensional $C^2$-manifold containing zero. Let $\mathcal E$ be parametrized by an injective map $\psi:U\subset \R^m \to \mathcal E$ with $\psi(0) = 0$,  where $\psi'(0)$ has full rank $m$. Assume $T:\V\to X$ is continuous, that $T(u) = 0$ for $u\in\mathcal  E$ and that there is $M\in \mathscr L(X)$ with spectral radius greater than one such that, for some $\sigma > 1$, 
\begin{equation}\label{e:T-sigma}
\|T(u) - Mu\| = \mathcal O(\|u\|^\sigma)\qquad \text{ as } u\to 0.
\end{equation}
Suppose further that $\partial_1\psi(0), ..., \partial_m\psi(0)\in \emph{\text{ker}} (M -\emph{\text{id}})$.
Then $u_* = 0$ is orbitally unstable with respect to $\mathcal E$ under iterations of $T$. More precisely, there is $\varepsilon_0 > 0$ such that for each $\delta > 0$ there are $u_\delta \in  \V$ with $\|u_\delta\|\leq \delta$ and $N\in \N$ such that $T^n(u_\delta)\in \V$ for $n= 1,..., N$ and $\emph{\text{dist}}(T^N(u_\delta), \mathcal E)\geq \varepsilon_0$. \end{lemma}

\begin{proof} \emph{Step 1.} Let $\alpha_0,\beta > 0$ such that $B_{5\alpha_0}(0) \subset \V$ and
\begin{equation}\label{e:2-lin}
 \|T(u_0) - M u_0\| \leq \beta \|u_0\|^\sigma, \qquad \|u_0\|\leq 5\alpha_0.
\end{equation}
There is an approximate eigenvalue $\lambda = r e^{\rmi\theta}$ with $r > 1$ and $\theta\in \R$ in the spectrum of $M$. Furthermore, there are $\eta, K > 0$ with $r+\eta < r^\sigma$ and $\|M^n\| \leq K(r+\eta)^n$ for all $n\geq 0$. In the sequel we choose $\alpha \in (0,\alpha_0)$ stepwise possibly smaller and smaller, only depending on $K,r,\eta,\beta,\psi$.

\emph{Step 2.} Let $\delta\in(0,\alpha)$ be given. As in the proof of  \cite[Lemma 5.1.4]{Henry} we find $N \in \N$ such that 
\begin{equation}\label{tststs}
\frac{\alpha}{r^N} \leq \delta, \qquad  |\sin (N\theta)|\leq \alpha,
\end{equation}
and $u,v\in X$ with $\|u\| = 1$ and $\|v\|\leq 1$ such that  
\begin{equation}\label{e:approx-ev}
 \|M^n (u+\text{i} v) - \lambda^n (u+\text{i} v)\| \leq \alpha. \qquad n= 1,...,N.
\end{equation}
Here the norm is actually the complexified one, i.e., $\|w_1+\text{i} w_2\| = \|w_1\| + \|w_2\|$ for $w_1,w_2\in X$.

Define $u_\delta := \frac{\alpha}{r^N}  u\in X$, such that $\|u_\delta \| =  \frac{\alpha}{r^N} \leq \delta$.  Let $n=1,...,N$ be given. Assume inductively that $\|T^{k}(u_\delta)\|\leq 5\alpha r^{k-N}$ for $k=0,...,n-1$. Then $T^n(u_\delta)$ is well-defined and as in the proof of \cite[Theorem 5.1.5]{Henry} we write
\begin{equation}\label{e:TM-lambda}
 T^n (u_\delta) - \lambda^n u_\delta =  \big( M^n u_\delta - \lambda^n u_\delta\big) + \sum_{k=0}^{n-1} M^{n-k-1}\big (T^{k+1}(u_\delta) - MT^{k}(u_\delta)\big ).
\end{equation}
Denote the right-hand side by $G_n + H_n$. We claim that
\begin{equation}\label{e:esti_G_H}
 \|G_n\| \leq \alpha ^2 r^{-N} +2\alpha |\sin(\theta n)| r^{n-N}, \qquad \|H_n\|\leq C_M\alpha^\sigma r^{n-N},
\end{equation}
where $C_M = \frac{5^\sigma K\beta }{r^\sigma-r-\eta}$ is independent of $n$. To see this, we use \eqref{e:approx-ev} to obtain
\begin{align}\label{e:Gn}
\|G_n\| &\, \leq  \frac{\alpha}{r^N}\big ( \|M^n u - (\text{Re}\,\lambda^n) u + (\text{Im}\,\lambda^n)\,v\|+  \|(\text{Im}\,\lambda^n)\,v\| +  \|(\text{Im}\,\lambda^n)\,u\| \big)\\\nonumber
&\, \leq  \frac{\alpha}{r^N}\big (\|\text{Re}((M^n -\lambda^n)(u+\text{i}v))\| + 2r^n|\sin(\theta n)| \big)\nonumber\\
&\, \leq  \alpha^2 r^{-N}+ 2\alpha |\sin(\theta n)|r^{n-N}.\nonumber
\end{align}
For the sum $H_n$ we use \eqref{e:2-lin}, that $\|T^{k}(u_\delta)\|\leq 5 \alpha r^{k-N} \leq 5\alpha_0$ for $k\leq n-1$ and that $r+\eta < r^\sigma$ to obtain 
\begin{align*}
\|H_n\| &\, \leq  \sum_{k=0}^{n-1} K(r+\eta)^{n-k-1} \beta (5\alpha r^{k-N})^\sigma\\
&\,  \leq \alpha^\sigma 5^\sigma K\beta  r^{\sigma(n-1-N)}\sum_{k=0}^{n-1}\Big( \frac{r+\eta}{r^\sigma}\Big)^{n-k-1}\leq C_M\alpha^\sigma r^{n-N}.
\end{align*}
This shows the claim \eqref{e:esti_G_H}. 

Now it follows from \eqref{e:TM-lambda}, \eqref{e:esti_G_H} and $\sigma > 1$ that $\|T^n(u_\delta)\| \leq 5\alpha r^{n-N},$ provided $\alpha\leq 1 $ is such that $C_M \alpha^{\sigma-1} \leq 1$. By induction, for all $n=0,...,N$ we obtain that $T^n(u_\delta)$ is well-defined and the estimates $\|T^n(u_\delta)\| \leq 5\alpha r^{n-N} $ and \eqref{e:esti_G_H} hold true.

\emph{Step 3}. As a consequence, for $\text{dist}(T^N(u_\delta), \mathcal E)$ we only have to consider $\zeta \in U$ such that $\|\psi(\zeta)\|\leq 10 \alpha$. Indeed, for $\|\psi(\zeta)\|> 10 \alpha$ we have $\|T^N(u_\delta) - \psi(\zeta)\| > 5\alpha$, but $\|T^N(u_\delta) - \psi(0)\| = \|T^N(u_\delta))\| \leq 5\alpha$. There is small $\tau_0 > 0$ such that 
\begin{equation}\label{e:S-para}
\psi(\zeta)  =   \psi'(0)\zeta + \rho(\zeta) 
,\qquad |\zeta|\leq \tau_0,
\end{equation}
where $\|\rho(\zeta)\| \leq C_\rho |\zeta|^2$ for a constant $C_{\rho}$ independent of $\zeta \in B_{\tau_0}(0)$. 
Since $\psi'(0)$ has full rank $m$, we have $C_{\psi'} = \min_{|\xi| = 1} \|\psi'(0)\xi\| > 0$ and we can choose $\tau_0$ such that $C_\rho\tau_0\leq C_{\psi'}/2$. Hence, with $\vartheta = 20/C_{\psi'}$ and small $\alpha$,  we obtain
$$
\|\psi(\zeta)\|\geq \|\psi'(0)\zeta\|- C_\rho|\zeta|^2> 10 \alpha \qquad \text{ for }\, \tau_0\geq  |\zeta|>\vartheta \alpha.
$$
Then, with these choices,
$$\text{dist}(T^N(u_\delta), \mathcal E) = \inf_{|\zeta|\leq \vartheta \alpha} \|T^N(u_\delta) - \psi(\zeta)\|.$$
\emph{Step 4}. Now let $|\zeta|\leq \vartheta \alpha$. Then \eqref{e:TM-lambda}, \eqref{e:S-para} and the estimates \eqref{e:esti_G_H} and $|\sin (N\theta)|\leq \alpha$ yield
\begin{align}
 \|T^N(u_\delta) - \psi(\zeta)\| &\, \geq  \|\lambda^N u_\delta - \psi'(0)\zeta\|-\| G_N\| - \|H_N\|- \|\rho(\zeta)\| \nonumber\\
 &\, \geq \|\alpha e^{\text{i}N\theta}u - \psi'(0)\zeta\| -3\alpha^2- C_M \alpha^\sigma - \vartheta^2 C_{\psi''} \alpha^2.\label{e:m3}
\end{align}
The vectors $u$ and $\psi'(0)\zeta$ are linearly independent if  $\alpha$ is sufficiently small. In fact, otherwise our assumption $\psi'(0)\zeta\in \text{ker}\,(M-\text{id})$ would imply that $Mu = u$. But as  in \eqref{e:Gn}, the estimate \eqref{e:approx-ev} then yields $|\lambda-1| =  \|\lambda u - Mu\| \leq \alpha^2 + 2\alpha,$
which is impossible for small $\alpha$. 

We conclude that $\| e^{\text{i}N\theta}u - \frac{1}{\alpha}\psi'(0)\zeta\|$ is bounded away from zero, uniformly for $|\zeta|\leq \vartheta \alpha$. Hence, decreasing $\alpha$ once more if necessary, we obtain from \eqref{e:m3} and $\sigma > 1$ that $\text{dist}(T^N(u_\delta), \mathcal E) \geq \varepsilon_0$, where $\varepsilon_0>0$ is a multiple of $\alpha$ independent of $\delta$. \end{proof}

Let us now  apply the lemma to  abstract quasilinear problems
\begin{equation}\label{e:quasi-3}
\partial_t u = A(u) u +F(u), \quad t > 0, \qquad u(0) = u_0.
\end{equation}
We denote by $\mathcal L(u_*) = A(u_*) + A'(u_*)[\cdot,u_*] + F'(u_*)$ the linearization of the right-hand side at $u_*$.

\begin{theorem} \label{t:instab-general} Assume the setting of either Theorem \ref{t:KPW-abstract} or Theorem \ref{t:lun-abstract}, and in addition that $A$ and $F$ are $C^2$.  Let $\mathcal E \subset \V\cap X_1$ be an $m$-dimensional $C^2$-manifold of equilibria of \eqref{e:quasi-3}, parametrized by $\psi: U\subset \R^m \to \mathcal E$, and let $u_* \in \mathcal E$ satisfy
\begin{itemize}
\item $\spec\,\mathcal L(u_*)\cap \{\emph{\text{Re}}\,\lambda > 0\} \neq \emptyset$,
\item $\partial_1\psi(\zeta_*),..., \partial_m \psi(\zeta_*) \in \text{\emph{ker}}\,\mathcal L(u_*)$ for $u_* = \psi(\zeta_*)$.
\end{itemize}
Then $u_*$ is orbitally unstable in $\V\subseteq \X$ with respect to $\mathcal E$.
\end{theorem}
\begin{proof} Shrink $\V$ around $u_*$ if necessary such that $t^+(u_0) \geq 1$ for each $u_0\in \V$. Let $\Phi_1:V\to \X$ be the time-one solution map for  \eqref{e:quasi-3}. Define $T(u_0) = \Phi_1(u_* + u_0) -(u_*+u_0)$ for $u_0$ close to $u_*$. Then $T$ is continuous, $T(u) = 0$ for $u\in \mathcal E\cap \V$, and $T$ satisfies \eqref{e:T-sigma} with $M = e^{\mathcal L(u_*)}\in \mathscr L(\X)$, as a consequence of Proposition \ref{prop:time-T} for the setting of Theorem \ref{t:KPW-abstract} and of \cite[Proposition 6.2]{Mey} for the setting of Theorem \ref{t:lun-abstract}. Moreover, $M$ has spectral radius larger than one by \cite[Corollary 2.3.7]{Lunardi1} and interpolation, and  $\partial_j \psi(\zeta_*)\in \text{ker}(M-\text{id})$ follows from the assumption. Thus Lemma \ref{l:Henry-lem} applies.
\end{proof}

Of course, Lemma \ref{l:Henry-lem} applies in any well-posedness setting for nonlinear parabolic problems.

\begin{appendix}
\section{Auxiliary results}
\subsection{Superposition operators}\label{a:sup}
We give some details for the properties of the nonlinear maps employed in the well-posedness results.

\begin{lemma} \label{lem:superpos-H1} Let $U_1,U_2\subset \R^N$ be open neighbourhoods of zero, let $a :\R\times U_1\to \R^N$ be $C^{k+3}$ and let $f:\R\times U_1\times U_2\to \R^N$ be $C^{k+2}$, with $f(\cdot,0,0) \in H^1$. Define 
$$A(u)v = (a(\cdot,u)v_x)_x, \qquad F(u) = f(\cdot,u,u_x).$$
Then there is an open subset $\V$ of $H^2$ such that $A \in C^k(\V, \mathscr L(H^3,H^1))$ and $F \in C^k(\V,H^1)$, and both maps are Lipschitz on bounded subsets of $\V$. One can take $\V = H^2$ if $U_1 = U_2 = \R^N$. At $\bu\in \V$, the derivatives are for $u\in H^2$ and $v\in H^3$ given by
$$A'(\bu)[u,v] =(\partial_2 a(\cdot,\bu)[u, v_x])_x, \qquad F'(u_0)v = \partial_2f(\cdot,\bu, \bu_x)v + \partial_3f(\cdot,\bu, \bu_x)v_x.$$
\end{lemma}
\begin{proof} Choose $\V \subseteq H^2$ such that for $u\in \V$ the closure of the images of $u,u_x \in H^1\subset \BC$ are uniformly contained in $U_1$ and $U_2$, respectively. Let $u\in \V$. For $h\in H^2$ we use $\|uh\|_{L^2}\leq \|u\|_{\BC}\|h\|_{L^2}$ and $\|u\|_{\BC}\leq C\|u\|_{H^1}$ to estimate
\begin{align*}
\|\partial_2 f(\cdot,u,u_x)h\|_{H^1}&\, \leq \|\partial_2f(\cdot,u,u_x)\|_{\BC}(\|h\|_{L^2} + \|h_x\|_{L^2} )\\
 + \|&f''(\cdot,u,u_x)\|_{\BC}(\|h\|_{L^2} + \|u_x\|_{L^2} \|h\|_{\BC} + \|u_{xx} \|_{L^2} \|h\|_{\BC})\\
 &\,\leq C(\|f'(\cdot, u,u_x)\|_{\BC} + \|f''(\cdot, u,u_x)\|_{\BC}\|u\|_{H^2}) \|h\|_{H^1}.
\end{align*}
In the same way we obtain
$$\|\partial_3 f(\cdot,u,u_x)h_x\|_{H^1}\leq C(\|f'(\cdot, u,u_x)\|_{\BC} + \|f''(\cdot, u,u_x)\|_{\BC}\|u\|_{H^2}) \|h\|_{H^2}.$$
Defining $F'(u) h = \partial_2 f(\cdot,u,u_x) h + \partial_3 f(\cdot,u, u_x) h_x$ we thus have $F'(u) \in \mathscr L(H^2, H^1)$, and that $u\mapsto F'(u)$ is bounded on bounded subsets of $\V$. If $h$ is small, then the pointwise identity
\begin{align*}
F(u+&\,h)-F(u)-F'(u)h   \\
&\,= \int_0^1 \int_0^1 \big(\partial_{22}f(\cdot, u + \tau s h, u_x)[h,\tau h] + \partial_{33}f(\cdot, u, u_x + \tau s h_x)[h_x,\tau h_x] \big)  \rmd \tau\rmd s
\end{align*}
and the same types of estimates as above yield
\[
\|F(u+h)-F(u)-F'(u)h\|_{H^1} \leq C(f,h)\|h\|_{H^2}^2,
\]
where $C(f,h)$ is bounded as $h\to 0$. These arguments and $f(\cdot,0,0)\in H^1$ yield $F(u)\in H^1$ for $u\in \V$ and the differentiability of $F$ in $\V$. The Lipschitz property follows from the boundedness of $F'$. Iteration for higher derivatives gives $F\in C^k$. The arguments apply to $u\mapsto a(u)$ on $H^2$ as well, which yields the assertion on $A$. \end{proof}

Note that if $f$ is independent of $u_x$, then the arguments from the proof above show that $f:H^1 \to H^1$ is smooth.

\begin{lemma}\label{lem:superpos-n} In the situation of Theorem \ref{t:quasiRDS-Rn}, assume in addition that $a$ and $f$ are $C^{k+2}$ for some $k\geq 0$. Let $A$ and $F$ be defined by
$$A(u)v =  \partial_i (a_{ij}(\overline{u}+ u) \partial_j v), \quad F(u) =   \partial_i (a_{ij}(\overline{u}+ u) \partial_j \overline{u}) + c_i \partial_i(\bar u +u)  + f(\overline{u}+u).$$
Then for all sufficiently large $p> 2$ there is is an open neighbourhood $\mathcal V\subset B_{2,p}^{2-2/p}$ of the zero function such that $F \in C^k(\V, L^2)$ and $A\in C^k(\mathcal V,  \mathscr L(H^2, L^2))$, and both maps are Lipschitz on bounded sets. One can take $\mathcal V = B_{2,p}^{2-2/p}$ if $U = \R^N$.
\end{lemma}
\begin{proof} Since $n\leq 3$, from Sobolev's embedding \eqref{e:Sob-Besov} we find $p > 2$ such that $B_{2,p}^{2-2/p} \subset H^{1,4}\cap \BC$. Then $\mathcal V$ can be chosen such that the image of $\bu+ u$ is strictly contained in $U$, uniformly in $u\in \mathcal V$. The regularity of $A$ and $F$ can be derived as in Lemma \ref{lem:superpos-H1}, using $F(0) = 0$. The need for $B_{2,p}^{2-2/p} \subset  H^{1,4}$ and thus also $H^2\subset H^{1,4}$ comes from the nonlinear gradient terms. Indeed, assume for simplicity that $\bu = 0$. Then  for $u_1,u_2\in B_{2,p}^{2-2/p}$ and $v\in H^2$ we can estimate
\begin{align*}
\|a_{ij}'(u_1)\partial_i &\,u_1 \partial_j v - a_{ij}'(u_2)\partial_i u_2 \partial_j v\|_{L^2} \leq \|a_{ij}'(u_1)\partial_i u_1 - a_{ij}'(u_2)\partial_i u_2\|_{L^4}\|\partial_j v\|_{L^4}\\
&\, \leq \big(\|a_{ij}'(u_1)\|_{\BC} \|u_1 - u_2\|_{H^{1,4}} + \|u_2\|_{H^{1,4}} \|\|a_{ij}'(u_1)-a_{ij}'(u_2)\|_{\BC}\big)\|v\|_{H^{1,4}},
\end{align*}
employing H\"older's inequality $L^4\cdot L^4\subset L^2$ in the first line. \end{proof}

\subsection{A commuting isomorphism for elliptic operators} The following auxiliary result for second order differential operators allows to transfer spectral properties from $L^2$ to $H^1$ by conjugation. 

\begin{lemma} \label{l:iso} Let $\alpha,\beta,\gamma\in\BC^1(\R, \R^{N\times N})$, and assume that $\alpha(x)$  is positive definite, uniformly in $x$. Then there is a continuous isomorphism $T:H^1 \to L^2$, which also maps $T:H^3 \to H^2 $ isomorphically, that commutes on $H^3$ with the operator $\varphi\mapsto \L\varphi := \alpha \varphi_{xx} + \beta \varphi_x + \gamma \varphi.$
\end{lemma}
\begin{proof} 
The isomorphism $T$ will be the square root of a shift of $\L$. The main point is to show that its domain for the realization on $H^2$ is $H^3$.

Denote by $\L_{L^2}$ the realization of $\L$ on $L^2$, with domain $H^2$. The properties of $\alpha$ together with \cite[Theorem 9.6]{AHS94} imply that there is $\omega> 0$ such that $B=\omega-\L_{L^2}$ is a (negative) sectorial operator and has a bounded holomorphic functional calculus of angle strictly smaller than $\frac{\pi}{2}$. In particular, $T:= B^{1/2}$ is a well-defined continuous isomorphism $D(B^{1/2}) \to L^2$, see \cite[Theorem 1.15.2]{Triebel}. The boundedness of the holomorphic calculus of $B$ implies that it has the property of bounded imaginary powers. Therefore, combining \cite[Lemma 4.1.11]{Lunardi2} with \cite[Theorem 1.15.3]{Triebel} (or \cite[Theorem 4.2.6]{Lunardi2}), we have $D(B^{1/2})  = [L^2, H^2]_{1/2}$, where $[\cdot,\cdot]_{1/2}$ denotes complex interpolation (see \cite{Bergh-Lof, Lunardi2, Triebel}). Since $[L^2, H^2]_{1/2} = H^1$ by \cite[Remark 2.4.2/2]{Triebel}, it follows that $T:H^1\to L^2$ is an isomorphism.

Next, we show that $T:H^3\to H^2$ is an isomorphism. Again by \cite[Theorem 1.15.2]{Triebel}, $T$ also maps isomorphically $D(B^{3/2}) \to D(B) = H^2$. We show that $D(B^{3/2}) = H^3$ as Banach spaces. By \cite[Lemma 4.1.16, Theorem 4.1.11]{Lunardi2} and the previous considerations we have 
$$D(B^{3/2}) = \{u\in D(B):B u\in D(B^{1/2})\} =  \{u\in H^2:\L u\in H^1\}.$$
For $u\in H^3$ we clearly have $\L u\in H^1$, hence $H^3\subseteq D(B^{3/2})$. Conversely, let $u\in H^2$ such that $\L u\in H^1$. Then $\alpha u_{xx} = \psi:= -\beta u_x - \gamma u+ \L u \in H^3$. By assumption, the coefficient $\alpha$ is pointwise invertible, with $\alpha^{-1} \in \BC^1$. Therefore $u_{xx} = \alpha^{-1} \psi \in H^1$, and so $u\in H^3$. We conclude that $D(B^{3/2}) = H^3$ as sets. Arguing as before, we get
$$\|u\|_{D(B^{3/2})} = \|u\|_{H^2} + \|\L u\|_{H^1} \leq C \|u\|_{H^3},$$
for a constant independent $C$ of $u$. Since we already know that $H^3$ is complete with respect to $\|\cdot\|_{D(B^{3/2})}$ and $\|\cdot\|_{H^3}$, the converse estimate follows from the open mapping theorem.

Finally, it follows from \cite[Theorem 4.1.6]{Lunardi2} that $\omega-\L_{L^2}$ and its square root $T$ commute on $H^3$. This implies that also $\L_{L^2}$ commutes with $T$.
\end{proof}

The assertion of the above lemma remains valid, with literally the same proof, if one replaces the $L^2$-setting by an $L^q$-setting, where $q\in (1,\infty)$.

\subsection{The time-one solution map} We use the implicit function theorem to prove that in the neighbourhood of an equilibrium the solution semiflow obained from Theorem \ref{t:KPW-abstract} for \eqref{e:quasi-abs} is as smooth as the right-hand side. See \cite[Theorem 3.4.4]{Henry} for the semilinear case, as well as  \cite[Theorem 8.3.4]{Lunardi1} and \cite[Theorem 4.1]{Amann3} for quasilinear frameworks. 

\begin{proposition} \label{prop:time-T}In the situation of Theorem \ref{t:KPW-abstract}, assume additionally that
$$A\in C^k(\mathcal V, \mathscr L(X_1,X_0)), \qquad F\in C^k(\mathcal V, X_0),$$
for some $k\in \N$. Let $u_*\in \mathcal V\cap X_1$ be an equilibrium of \eqref{e:quasi-abs}, i.e., $A(u_*)u_* + F(u_*) = 0$. Then for any $\tau>0$ there is a neighbourhood $\mathcal U \subseteq \mathcal V$ of $u_*$ such that the time-$\tau$ map $u_0\mapsto \Phi_\tau(u_0) = u(\tau;u_0)$ for the solution semiflow for  \eqref{e:quasi-abs} is well-defined and belongs to $C^k(\mathcal U, \mathcal X)$. Moreover, let $\mathcal L_* = A(u_*) + A'(u_*)[\cdot,u_*] + F'(u_*)$. Then  $\Phi_\tau' (u_*)= e^{\tau\mathcal L_*}.$
\end{proposition}
\begin{proof} We assume $\mathcal V = \X$. Set $\E_1 = H^{1,p}(0,\tau; X_0) \cap L^p(0,\tau; X_1)$ and $\E_0 = L^p(0,\tau; X_0),$ and consider 
$$\Psi :\E_1\times \mathcal X \to \E_0\times \mathcal X, \qquad \Psi(u,u_0) = (\partial_t u - A(u)u -F(u), u(0) -u_0).$$ Note that $u\in \E_1$ solves \eqref{e:quasi-abs} on $(0,\tau)$ with initial value $u_0\in \mathcal X$ if and only if $\Psi(u,u_0) = (0,0)$.
Consider $u_*$ as an element of $\E_1$. Then $\Psi(u_*,u_*) =(0,0)$. The assumptions on $A$ and $F$ imply $\Psi\in C^k(\E_1\times \mathcal X, \E_0\times \mathcal X)$ and
$$D_{1}\Psi(u_*,u_*)v = (\partial_t v - \L_* v, v(0)),\qquad v\in \E_1.$$
From the proof of Theorem \ref{t:KPW-abstract} we know that $-A(u_*)$ enjoys maximal $L^p$-regularity. The linear operator $A'(u_*)[\cdot,u_*] + F'(u_*)$ is continuous from $\mathcal X = (X_0,X_1)_{1-1/p,p}$ to $X_0$, i.e., it is of lower order. Thus $-\mathcal L_*$ has maximal $L^p$-regularity as well, see \cite[Theorem 6.2]{Dore00}. In other words, $D_1\Psi(u_*,u_*)\in \L (\E_1, \E_0\times \X)$ is an isomorphism. This gives a neighbourhood $\mathcal U$ of $u_*$ in $\mathcal X$ such that $u_0\mapsto u(\cdot;u_0)$ belongs to $C^k(\mathcal U, \E_1)$, where $u(\cdot;u_0)$ is the solution of \eqref{e:quasi-abs} on $(0,\tau)$. Moreover, for $v_0\in \X$ we  differentiate $\Psi(u(\cdot;u_0),u_0)=0$ in $u_*$ to get that
$$D_{u_0} u(\cdot;u_*)v_0= - D_1\Psi(u_*,u_*)^{-1} D_2\Psi(u_*,u_*)v_0 = - D_1\Psi(u_*,u_*)^{-1} (0,-v_0)$$
is the unique solution $v\in \E_1$ of $\partial_t v -\L_* v = 0$ on $(0,\tau)$ with $v(0) = v_0$, i.e., $D_{u_0} u(\cdot;u_*) = e^{\cdot\L_*}$. Finally, the trace at time $\tau$ is linear and continuous as a map $\E_1\to \X$, see \cite[Theorem 1.14.5]{Triebel}. Applying this to $u(\cdot;u_0)$ gives the assertion for $\Phi_\tau$. \end{proof}
\end{appendix}

\end{document}